\documentclass[12pt]{amsart}

\usepackage{amssymb,amsthm}
\usepackage{amsmath,amstext}
\usepackage{amscd}
\usepackage{mathrsfs}
\usepackage{bbm}
\usepackage{bbold}
\usepackage{graphicx}
\usepackage{psfrag}
\usepackage{multirow}
\usepackage{hyperref}
\usepackage[usenames,dvipsnames,svgnames,table]{xcolor}
\usepackage{rotating}
\usepackage{ifthen}
\usepackage{epic}
\usepackage{enumitem}
\usepackage{subcaption}
\usepackage{hhline}

\usepackage{tikz}
\usetikzlibrary{matrix,positioning}
\usetikzlibrary{arrows}
\usetikzlibrary{decorations.markings}
\usetikzlibrary{shapes.geometric}
\usetikzlibrary{shapes.misc}
\usetikzlibrary{calc}
\usetikzlibrary{cd}

\newtheorem{theorem}{Theorem}[section]
\newtheorem{lemma}[theorem]{Lemma}
\newtheorem{prop}[theorem]{Proposition}
\newtheorem{cor}[theorem]{Corollary}
\newtheorem{conjecture}[theorem]{Conjecture}

\theoremstyle{remark}
\newtheorem{remark}[theorem]{Remark}

\numberwithin{equation}{section}

\newcommand{\abs}[1]{\lvert#1\rvert}
\newcommand{\ds}{\displaystyle}
\renewcommand{\neg}[1]{\overline{#1}}
\newcommand{\al}{\alpha}
\newcommand{\1}{\neg{1}}

\newcommand{\bw}{\ensuremath{{1\atop\neg{1}}}}
\newcommand{\bq}{\ensuremath{{1\atop ?}}}
\newcommand{\wq}{\ensuremath{{\neg{1}\atop ?}}}
\newcommand{\wb}{\ensuremath{{\neg{1}\atop 1}}}
\newcommand{\bb}{\ensuremath{{1\atop 1}}}
\newcommand{\ww}{\ensuremath{{\neg{1}\atop\neg{1}}}}
\newcommand{\xx}{\ensuremath{{0 \atop 0}}}
\newcommand{\dg}{\ensuremath{{1 \atop }|{ \atop \neg{1}}}}
\newcommand{\st}{\ensuremath{{* \atop *}}}
\makeatletter
\newcommand*\dashedline{\rotatebox[origin=c]{90}{$\dabar@\dabar@\dabar@$}}
\makeatother

\newcommand{\aver}[1]{\langle #1 \rangle}
\newcommand{\catalan}[1]{C_{#1}}
\newcommand{\ballot}[2]{C^{#1}_{#2}}
\newcommand\Wtilde{\stackrel{\sim}{\smash{W}\rule{0pt}{1.1ex}}}

\newcommand{\dstar}{$D^*$-TASEP}

\newcommand{\dt}{\D-TASEP}
\newcommand{\ct}{\chC-TASEP}
\newcommand{\bt}{\hatB-TASEP}
\newcommand{\dmt}{\D-MultiTASEP}
\newcommand{\cmt}{\chC-MultiTASEP}
\newcommand{\bmt}{\hatB-MultiTASEP}

\newcommand{\hatB}{\ensuremath{{B}}}
\newcommand{\chB}{\ensuremath{{\check B}}}
\newcommand{\hatC}{\ensuremath{{C}}}
\newcommand{\chC}{\ensuremath{{\check C}}}
\newcommand{\D}{\ensuremath{{D}}}

\DeclareMathOperator{\Col}{Col}
\DeclareMathOperator{\Row}{Row}
\DeclareMathOperator{\Hu}{UHook}
\DeclareMathOperator{\Hd}{DHook}

\begin{document}

\title[Limiting directions for random walks in Weyl groups]
{Limiting directions for random walks in classical affine Weyl groups}

\author{Arvind Ayyer}
\address{Arvind Ayyer, Department of Mathematics, Indian Institute of Science, Bangalore - 560012, India}
\email{arvind@iisc.ac.in}

\author{Svante Linusson}
\address{Svante Linusson, Department of Mathematics, Royal Institute of Technology - KTH, Stockholm, SE-10044 Sweden.}
\email{linusson@kth.se}

\author{Samu Potka}
\address{Samu Potka, Department of Mathematics, Royal Institute of Technology - KTH, Stockholm, SE-10044 Sweden.}
\email{potka@kth.se}



\date{\today}

\begin{abstract}
Let $W$ be a finite Weyl group and $\Wtilde$ the corresponding affine Weyl group. A random element of $\Wtilde$ can be obtained as a reduced random walk on the alcoves of $\Wtilde$. By a theorem of Lam (Ann. Prob. 2015), such a walk almost surely approaches one of $|W|$ many directions. We compute these directions when $W$ is $B_n$, $C_n$ and $D_n$ and the random walk is weighted by Kac and dual Kac labels. This settles Lam's questions for types $B$ and $C$ in the affirmative and for type $D$ in the negative. The main tool is a combinatorial two row model for a totally asymmetric simple exclusion process called the \dstar{}, with four parameters. By specializing the parameters in different ways, we obtain TASEPs for each of the Weyl groups mentioned above. Computing certain correlations in these TASEPs gives the desired limiting directions.
\end{abstract}

\maketitle

\section{Introduction}

Let $W$ be a finite Weyl group and $\Wtilde$ the corresponding affine Weyl group. In \cite{lam-2015}, Lam studied large random elements of $\Wtilde$ obtained by multiplication by a randomly chosen simple reflection at each step under the condition that the expression stays reduced. This can also be described as a random walk on the alcoves of the affine Weyl group. In each step, the reduced random walk will cross into a new alcove based on the condition that it must never cross a hyperplane that has already been crossed. See Figure \ref{F:TypeB} for an example of type $B$.
This process is a Markov chain and the walk will after a finite number of steps 
almost surely be confined to one chamber of the underlying finite Weyl group. In that chamber it will by the law of large numbers go in a certain direction. Lam proved \cite{lam-2015} that the probability that such a reduced walk will end up in an chamber corresponding to an element $w\in W$ is given by the stationary distribution $\pi_W$ of a certain finite state Markov chain with rules defined by the algebra of the group $W$. He also proved a formula for the exact direction of the walk in terms of $\pi_W$. 

\begin{figure}
\centering
\includegraphics[width=.5\textwidth]{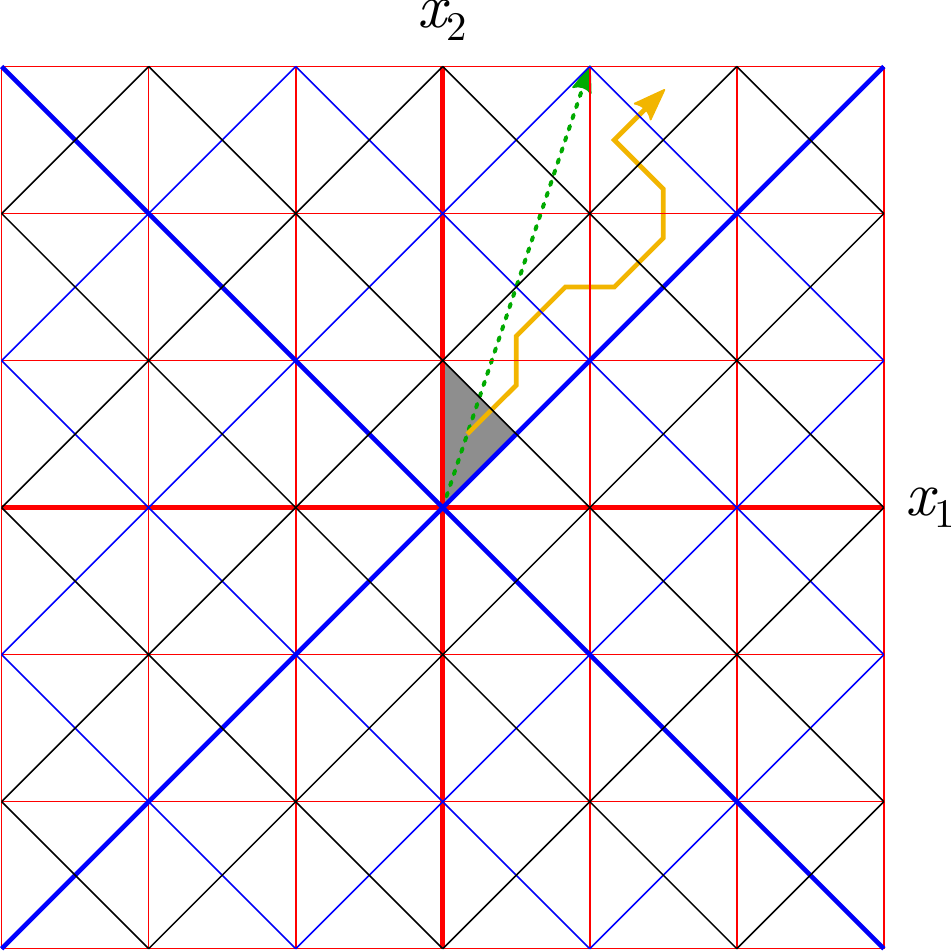}
\caption{Example of a random walk in $\tilde B_2$. The red lines ($x_i=k$) correspond to the root $\alpha_0$, changing the sign of the first element, the blue lines ($x_1\pm x_2=2k$) to $\alpha_1$, swapping the two elements, 
and the black lines ($x_1\pm x_2=2k+1$) to $\theta$, replacing $w_1,w_2$ with $w_4,w_3$ in $\tilde B_2$ which means swapping and changing the signs of the last two elements in the \bmt{}. 
The yellow path corresponds to the word $\dots\tilde s_1\tilde s_2\tilde s_0\tilde s_2\tilde s_0\tilde s_1\tilde s_0\tilde s_2$. After seven steps it is confined to the chamber of the identity of $B_2$. The (green) dotted line is the limiting direction of the reduced random walk.}
\label{F:TypeB}
\end{figure}

For the Weyl group of type $A$, Lam conjectured that this limiting direction is given by the sum of all positive roots.
This was proved to be the case by 
the second and third authors~\cite{ayyer-linusson-2016} by computing correlations in the multispecies TASEP on the ring, a model that had already been considered in the physics literature with completely different motivation~\cite{evans-ferrari-mallick-2009, prolhac-etal-2011}. 
The stationary distribution of the multispecies TASEP had previously been independently computed by Ferrari and Martin \cite{FM07} where they introduced the so-called {\em multiline queues} as a combinatorial tool to understand the stationary distribution. 

In this article, we study the corresponding exclusion processes for the classical affine Weyl groups of type $B$, $C$ and $D$, which will give us limiting directions of the reduced random walks for these groups. Here a natural thing to do is to weight the walks with the so-called {\em Kac labels} or {\em dual Kac labels}~\cite{kac-1990} for each of the types; see Table \ref{T:Kac}. 
For $\tilde A_n$ the limiting direction was proved to be equal to the sum of positive roots of $A_n$ in \cite {ayyer-linusson-2016} as conjectured in \cite{lam-2015}. 
It was also claimed~\cite{lam-2015} that, using the Kac labels as weights, 
a similar statement held true for type $B$ and was close to being true for the other types.
 We prove here that this is true for type $B$ (Section \ref{ss:hatB}), but not true for type $D$ (Section \ref{ss:D}). 
However, the expression looks similar to the sum of positive roots for type $C$ using the dual Kac labels instead; see Remark~\ref{rem:similar}. Strangely, our computations suggest that the limiting directions for types $C$ and $D$ using Kac labels are closely related; see 
Remark~\ref{rem:cequald}. 

As a combinatorial tool, we use a two-row model for a more general Markov chain called the \dstar{}, studied in \cite{AALP-2019}.
This model has four parameters, and by specializing these parameters we get the TASEPs needed for each of the Weyl groups with Kac and dual Kac labels we consider here. The three TASEPs studied in detail, named the \bmt{}, \cmt{} and \dmt{}, are interesting in their own right. 
In particular, the \bmt{} has interesting properties and certain two-point correlations in that chain have the same curious independence property that was proven in type $A$; see Conjecture \ref{conj:b-corr}.
For the multiTASEPs of types $\hatC$, $\chC$ we can lump to a TASEP that has previously been studied in the literature~\cite{arita-2006,als-2009} and we can use the results therein to find the limiting direction in this case.

The paper is organized as follows. In Section \ref{sec:Background} we give the background and conventions for the Weyl groups and the root systems used. In Section \ref{sec:results} we state our main results and conjectures. In Section~\ref{sec:models} we define the TASEPs used for each Weyl group and in Section~\ref{sec:color}, we describe how we project them to TASEPs with only two species.
In Section \ref{sec:duchi-schaeffer} we give a two row combinatorial interpretation of the \dstar{}, which is a modification 
of the two-row model given by Duchi and Schaeffer~\cite{duchi-schaeffer-2005}. This is used in Section~\ref{sec:part} use to compute the partition functions and two-point correlations needed for proving the limiting direction in each case.

\section{Background}
\label{sec:Background}

Weyl groups are finite reflection groups studied extensively in both Lie theory and combinatorics. A Weyl group $W$ is generated by a number of simple reflections $\{s_0,\dots,s_{n-1}\}$ corresponding to simple roots. In the affine group $\Wtilde$ the corresponding reflections are written as $\tilde s_0,\dots,\tilde s_{n-1}$ and accompanied by one more reflection $\tilde s_n$ in the so-called highest root $\theta$ which together generate the group. Lam~\cite{lam-2015} studied the shape of a random semi-infinite element of $\Wtilde$. His model of randomness is the following: at each step, a word is multiplied on the left by a simple reflection subject to the condition that the word is reduced. This corresponds to a reduced random walk on the alcoves of the hyperplane arrangement corresponding to $\Wtilde$, where reduced means that no hyperplane is crossed twice. For an example of type $B$, see Figure \ref{F:TypeB}. Such a walk will, with probability one, be confined to one of the chambers of $W$ and tend to go in a specific direction, for which Lam gave a formula
 (see Theorem \ref{T:lam}) in terms of the stationary distribution $\pi_W$ of a corresponding TASEP defined on the
 underlying group $W$. A step $\tilde s_i$, $0\le i\le n-1$, in the random walk corresponds to $s_i$ in the TASEP, and $\tilde s_n$ corresponding to the reflection $r_\theta$ in the longest root is also easily interpreted in each case; see below.

In this article, we will determine this direction for the classical root systems $\tilde B_n$, $\tilde C_n$ and $\tilde D_n$. For the group $\tilde A_n$ of permutations this was done in \cite{ayyer-linusson-2016}.
We will use the combinatorial description in terms of signed permutations for both the finite and affine groups as presented in \cite[Chapter 8]{bjorner-brenti-2005}. 
We think of elements of $\Wtilde$ in the window notation: $w = [w_1, \dots, w_n]$ means that $w(i) = w_i$ for $1\leq i \leq n$. For an element $(w_1,\dots,w_n)$ in the finite $W$, we have $-n\le w_i\le n$. 
Taking each coordinate modulo $2n+1$ gives an element in $W$ from one in $\Wtilde$. Recall that $B_n = C_n$ consist of all possible signed permutations, whereas $D_n$ consists of only those signed permutations which have an even number of negative signs. We will use both $-i$ and $\neg i$ to denote a signed
 element. In all the finite groups $s_i$, $1\le i\le n-1$, interchanges $w_i$ and $w_{i+1}$. The reflection $s_0$ changes the sign of $w_1$ in $B_n$ and $C_n$, whereas in $D_n$ it swaps $w_1$ and $w_2$ and changes the sign of both of them. The reflection $r_\theta$ changes the sign of $w_n$ in $C_n$. In $B_n$ and $D_n$, the reflection $r_\theta$ swaps $w_{n-1}$ and $w_n$ and changes the sign of both of them. Note how this follows from the roots listed below.

We will pick explicit roots in $\mathbb{R}^n$, and let $W$ act on $\mathbb{R}^n$ by $w\cdot e_i = e_{w^{-1}(i)}$, where we define $e_{-i} = -e_i$.
Lam's formula is as follows.

\begin{theorem}[{\cite[Theorem 2]{lam-2015}}]
\label{T:lam}
The direction for a reduced random walk in the Weyl group $\Wtilde$, which is confined to the identity chamber, is parallel to 
\[
\psi = \sum_{w:r_\theta w > w} \pi_W(w) \; w^{-1}\cdot \theta. 
\]
\end{theorem}

To determine if $r_\theta w>w$ we use the combinatorial formulas for inversions; see \cite[Chapter 8]{bjorner-brenti-2005}. 
The explicit formula for $\psi$ is given for $C_n$ in \eqref{eq:limC} and for $B_n$ and $D_n$ in \eqref{eq:limB}.

We next list the Coxeter group conventions we employ in this article for the classical types. We note that these conventions are slightly different from those used by Bj\"orner and Brenti~\cite{bjorner-brenti-2005}, but are used in \cite[Chapter 17]{boroviks-2010}, for example.
Lam's proofs are uniform for all types and therefore, do not need to assume any such choices.
For each type ($C$, $B$ and $D$), we list the choice of roots, simple roots and highest root in Table~\ref{tab:root-data}.

\begin{table}[htbp!]
\centering
\begin{tabular}{|p{0.04 \textwidth} | p{0.9 \textwidth}|}
\hline
{$C_n$} & 
{\bf Roots}: $\pm e_i \pm e_j$ for $i < j$, together with $\pm 2e_i$. \\
& {\bf Simple roots}: $(\alpha_0, \dots, \alpha_{n-1}) = (2e_1$, $e_2 - e_1$, \dots, $e_n - e_{n-1})$. \\
& {\bf Highest root}: $\theta = 2e_n$. \\
\hline
{$B_n$} & 
{\bf Roots}: $\pm e_i \pm e_j$ for $i < j$, together with $\pm e_i$. \\
& {\bf Simple roots}: $(\alpha_0, \dots, \alpha_{n-1}) = (e_1$, $e_2 - e_1$, \dots, $e_n - e_{n-1})$. \\
& {\bf Highest root}: $\theta = e_{n-1} + e_n$.\\
\hline
{$D_n$} & 
{\bf Roots}: $\pm e_i \pm e_j$ for $i < j$. \\
& {\bf Simple roots}: $(\alpha_0, \dots, \alpha_{n-1}) = (e_1 + e_2, e_2 - e_1, e_3-e_2, $\\
& $\dots, \allowbreak e_n - e_{n-1})$. \\
& {\bf Highest root}: $\theta = e_{n-1} + e_n$. \\
\hline
\end{tabular}
\caption{Root data for types $C$, $B$ and $D.$}
\label{tab:root-data}
\end{table}

\subsection{Kac and dual Kac labels as weights}

\begin{table}[htbp!]
\centering
\begin{tabular}{| l | c | c | c | c| c | }
\hline
Type & $a_0$ &$a_1$ &$\dots a_i\dots $ &$a_{n-1}$ &$a_{n}$  \\ 
\hline &&&&&\\[-0.3cm]
$A = \check{A}$ & 1 & 1 & 1 & 1 & 1 \\
\hline &&&&&\\[-0.3cm]
\hatB & 2 & 2 & 2 & 1 & 1 \\
\hline &&&&&\\[-0.3cm]
\hatC & 1 & 2 & 2 & 2 & 1 \\
\hline &&&&&\\[-0.3cm]
$D = \check{D}$ & 1 & 1 & 2 & 1 & 1 \\
\hline
 \end{tabular}
 \medskip

 \begin{tabular}{| l | c | c | c | c| c | }
\hline
Type & $\check a_0$ &$\check a_1$ &$\dots \check a_i\dots $ &$\check a_{n-1}$ &$\check a_{n}$  \\ 
\hline &&&&&\\[-0.3cm]
\chB & 1 & 2 & 2 & 1 & 1 \\
\hline &&&&&\\[-0.3cm]
\chC & 1 & 1 & 1 & 1 & 1 \\
\hline
 \end{tabular}
 \medskip
 
\caption{Kac labels $a_i$ and dual Kac labels $\check{a}_i$ for the different classical infinite families of Weyl groups.}
\label{T:Kac}
\end{table}

For the symmetric group $A_n$ the random walk that assigned equal probability to each possible step in each situation was the natural choice and corresponds to the multispecies TASEP on a ring, which was used to 
compute the limiting direction in \cite{ayyer-linusson-2016}. As Lam noted~\cite[Remark 5]{lam-2015}, it makes sense to adjust the weights according to the Kac labels for the other classical groups. They are defined as follows.
For positive integers $a_i$, $0\le i\le n-1$, write $\theta=\sum_{i=0}^{n-1} a_i\alpha_{i}$. See Table \ref{T:Kac} for the values of $a_i's$ for each type, where $a_{n}=1$ for the highest root. We will then let $a_i$ be the 
rate with which the reduced walk takes a step corresponding to $\tilde s_{i}$.

As Lam also pointed out, another natural choice of weights for the reduced walk is the dual Kac labels $\check a_i$, for reasons related to the topology of the affine Grassmanian \cite[Section 5.5]{lam-2015}. They are 
also presented in Table \ref{T:Kac}. For these labels, see \cite[Tables Aff 1 and Aff2]{kac-1990}. The duality of the Dynkin diagrams amounts to reversing the arrows and since type $A$ and $D$ are simply laced, they are self-dual. We will write the random walks with these weights as \hatB, \chB, \hatC, \chC \ and  \D\  respectively. For each of these cases one can define the corresponding TASEP. In Section \ref{sec:models} we will define 
in detail the TASEPs corresponding to the three cases $\chC$, $\hatB$, and $\D$, and omit the details for $\hatC$ and $\chB$. They can be defined using the same methods.

\section{Main Results}\label{sec:results}

The overall strategy of proof will be similar to the proof of Lam's conjecture for type $A$~\cite{ayyer-linusson-2016}. We will recast the finite version of Lam's chain on the Weyl chambers in terms of an exclusion process. 

\subsection{Type \chC}
\label{ss:chC}

For the calculations of the limiting directions in type \chC, we will appeal to the exclusion process known as the \cmt{} (see Section~\ref{sec:cmodel}). We will fix $n$ and let $\pi_\chC$ denote the stationary distribution of the \cmt{} on $n$ sites.
In this section, we denote by ${\aver i}$ the $\pi_\chC$-probability that the last site in the \cmt{} on $n$ sites is occupied by a particle of species $i$ for $1 \leq i \leq n$. We will appeal to the formula for $\aver i$ given in Theorem~\ref{thm:typec-dens}.

\begin{theorem}
\label{thm:lim-dir-c}
The limiting direction of Lam's random walk on the alcoves of the affine Weyl group $\tilde C_n$ with probabilities weighted by dual Kac labels is given by
\[
\sum_{i=1}^n (2i+1) e_i.
\]
\end{theorem}

\begin{proof} 
By Theorem \ref{T:lam}, the limiting direction is 
\begin{align}\label{eq:limC}
\sum_{\substack{w \\ r_\theta w  > w}} \pi_\chC (w) \; w^{-1} \cdot \theta 
= \sum_{\substack{w \\ w_n > 0}} \pi_\chC(w) \; e_{w_n}
=\sum_{i=1}^n {\aver{i}} \; e_i.
\end{align}
The first equality follows from the fact that the reflection $r_\theta$ corresponds to changing the sign of the last 
position in $w$ and the formula for the length in $C_n$ (which is the same as $B_n$)
(see \cite[Proposition 8.1.1]{bjorner-brenti-2005}) in terms of the number of inversions gives that the length will increase if the last position becomes negative. 
The second equality follows from the definition of $\aver i$.
Theorem~\ref{thm:typec-dens} now gives 
\[
{\aver i}=\frac{2i+1}{2n(2n+1)},
\]
determining the limiting direction up to a constant.
\end{proof}

\begin{remark}
\label{rem:similar}
Note that the sum of positive roots for $C_n$ (from Table~\ref{tab:root-data}) is
\[
\sum_{1 \leq i < j \leq n} (e_j - e_i) + 
\sum_{1 \leq i < j \leq n} (e_j + e_i) + 2 \sum_{i=1}^n e_i
= \sum_{i=1}^n 2i \; e_i.
\]
This is very similar to, but different from, the result of Theorem~\ref{thm:lim-dir-c}.
\end{remark}

\subsection{Type \hatB} 
\label{ss:hatB}

For the calculations of the limiting directions in type \hatB, we will appeal to the exclusion process known as the \bmt{} (see Section~\ref{sec:bmodel}). We will fix $n$ and let $\pi_B$ denote the stationary distribution of the \bmt{} on $n$ sites.
In this section, we denote by $\aver {i,j}$ the $\pi_B$-probability that the last two sites in the \bmt{} on $n$ sites are occupied by particles of species $i$ and $j$ respectively. Here, we take $-n \leq i,j \leq n$ and denote $-k$ by $\neg{k}$. We will appeal to the formulas for certain sums of $\aver {i,j}$ given in Lemma \ref{L:Corr-sums-B}.

\begin{theorem}
\label{thm:lim-dir-b}
The limiting direction of Lam's random walk on the alcoves of the affine Weyl group $\tilde B_n$ with probabilities weighted by Kac labels is given by
\[
\sum_{i=1}^n (2i-1) e_i.
\]
\end{theorem}

\begin{proof} By Theorem \ref{T:lam}, the limiting direction is 
\begin{align*}
 \sum_{\substack{w \\ r_\theta w > w}} \pi_\hatB(w) \; w^{-1}\cdot (e_{n-1} + e_n) 
= \sum_{\substack{w \\ r_\theta w > w}} \pi_\hatB(w) (e_{w_{n-1}} + e_{w_n}).
\end{align*}
Just as in type \chC{},  we first use the interpretation of $r_\theta$ as simultaneously switching position and sign for $w_{n-1}, w_n$.
Then we use the length of $w$ in terms of its inversions in $B_n$~\cite[Proposition 8.1.1]{bjorner-brenti-2005} and perform a case analysis to determine how the $r_\theta$ move affects the length of $w$ to obtain
\begin{align*}
& \sum_{\substack{w \\ r_\theta w > w}} \pi_\hatB(w) (e_{w_{n-1}} + e_{w_n}) \\
=& \sum_{1 \leq i < j \leq n} \bigg(
\sum_{\substack{w \\ (w_{n-1},w_n) = (j,i)}} \pi_\hatB(w)(e_i+e_j) 
+ \!\sum_{\substack{w \\ (w_{n-1},w_n) = (j,\neg{i})}} \pi_\hatB(w)(-e_i + e_j) \\
&+\sum_{\substack{w \\ (w_{n-1},w_n) = (i,j)}} \pi_\hatB(w)(e_i + e_j) +\sum_{\substack{w \\ (w_{n-1},w_n) = (\neg{i},j)}} \pi_\hatB(w)(-e_i + e_j) \bigg) \\
=&\sum_{1\leq i < j \leq n} \big( {\aver{j,i}} (e_i+e_j) + {\aver{j,\neg{i}}} (-e_i+e_j) \\
&+{\aver{i,j}} (e_i+e_j) + {\aver{\neg{i},j}} (-e_i+e_j) \big).
\end{align*}
This can be further simplified to
\begin{multline}
\label{eq:limB}
\sum_{i=1}^n e_i \left(\sum_{j=i+1}^n {\aver{j,i}} - {\aver{j,\neg{i}}} + {\aver{i,j}} - {\aver{\neg{i}, j}} \right) \\
+ \sum_{j=1}^n e_j \left(\sum_{i=1}^{j-1} {\aver{j,i}} + {\aver{j,\neg{i}}} + {\aver{i,j}} + {\aver{\neg{i}, j}} \right).
\end{multline}
With notation from Section \ref{sec:limB} we can rewrite the coefficient for $e_k$ as 
\begin{multline}
\sum_{j=-k+1}^n \big( {\aver {j,k}} + {\aver {k,j}} \big) - \sum_{j=k+1}^n \big( {\aver {j,-k}} + {\aver {-k,j}} \big) \\
 = \Row_k(n)-\Hd_k(n)+\Col_k(n)-\Hu_k(n).
\end{multline}
Plugging in the formulas from Lemma \ref{L:Corr-sums-B} for $1\le k\le n$, this simplfies to $\frac{2k-1}{n(2n-1)}$, which proves the result.
\end{proof}

\begin{remark}
Note that the sum of positive roots for $B_n$ (from Table~\ref{tab:root-data}) is
\[
\sum_{1 \leq i < j \leq n} (e_j - e_i) + 
\sum_{1 \leq i < j \leq n} (e_j + e_i) + \sum_{i=1}^n e_i
= \sum_{i=1}^n (2i -1) \; e_i.
\]
This is exactly the same as the result of Theorem~\ref{thm:lim-dir-b} up to an overall scaling. This proves Lam's claim~\cite[Remark 5]{lam-2015}.
\end{remark}

We note that the proof of Theorem~\ref{thm:lim-dir-b} uses certain sums of $\aver{i,j}$'s,
but our techniques do not allow us to determine all $\aver{i,j}$'s.
However, we are able to give conjectures for all two-point correlations. 
Since we have the identity $\aver{i,\neg j} = \aver{i,j}$ from Theorem~\ref{T:Bprob},
it will suffice to conjecture $\aver{i,\neg j}$ and $\aver{\neg i,\neg j}$.
From numerical data for small values of $n$, we arrive at the the following.

\begin{conjecture} 
We conjecture the following two-point correlations for the last two positions in the \bmt{}.

\label{conj:b-corr}
\begin{enumerate}

\item For $3 \leq i \leq n, 1 \leq j \leq i-2$, 
\[
\aver{ \neg{i}, \neg{j} } = \frac 1{(2n)^2}.
\]

\item For $1\le j\le n-1$, 
\[
{\aver{\overline{j+1},\neg j}} =\frac{1}{(2n)^2}+\frac{n^2-j^2}{4n^2(2n-1)}.
\]

\item For $1 \leq i \leq n-1, i+1 \leq j \leq n$,  
\[
\aver{ \neg{i}, \neg{j} } = \frac{j-i}{2n^2 (2n-1)},
\] 
and for $1 \leq i \leq n-2, i+2 \leq j \leq n$, 
\[
\aver{ i, \neg{j} } = \frac{i+j-1}{2n^2 (2n-1)}.
\] 

\item For $1\le j\le n-1$, 
\[
{\aver{j,\overline {j+1}}}=\frac{j(n^2-j^2+2n-2)}{2n^2(2n-1)(n-1)}.
\]

\item For $2 \leq i \leq n, 1 \leq j \leq i-1$,  
\[
\aver{ i, \neg{j} } = \frac{3(i-j)(i+j-1)}{4n^2 (2n-1)(n-1)}.
\]
\end{enumerate}
\end{conjecture}

In particular, cases (1) and (3) of Conjecture~\ref{conj:b-corr} resemble the results for type $A$ in \cite[Theorem 4.2]{ayyer-linusson-2016}. Notice that the correlations in case (1) behave as if the particles were independent.
We illustrate these with data for $n=4$ in Table~\ref{tab:bvalues-n=4}. These cases are marked in boldface.

\begin{table}[htbp!]
\begin{equation*}
\begin{array}{c|cccc}
i \backslash j & \neg{4} & \neg{3} & \neg{2} & \neg{1} \\[0.1cm]
\hline
&&&& \\[-0.3cm]
\neg{4} & 0 & \frac{1}{32} & \mathbf{\frac{1}{64}} & \mathbf{\frac{1}{64}} \\[0.1cm]
\neg{3} & \mathbf{\frac{1}{224}} & 0 & \frac{19}{448} & \mathbf{\frac{1}{64}} \\[0.1cm]
\neg{2} & \mathbf{\frac{2}{224}} & \mathbf{\frac{1}{224}} & 0 & \frac{11}{224} \\[0.1cm]
\neg{1} & \mathbf{\frac{3}{224}} & \mathbf{\frac{2}{224}} & \mathbf{\frac{1}{224}} & 0 \\[0.1cm]
1 & \mathbf{\frac{4}{224}} & \mathbf{\frac{3}{224}} & \frac{1}{32} & 0 \\[0.1cm]
2 & \mathbf{\frac{5}{224}} & \frac{3}{56} & 0 & \frac{1}{224} \\[0.1cm]
3 & \frac{13}{224} & 0 & \frac{1}{112} & \frac{3}{224} \\[0.1cm]
4 & 0 & \frac{3}{224} & \frac{5}{224}& \frac{3}{112}
\end{array}
\end{equation*}
\caption{The values of $\aver{ i, \neg j }_{\hatB}$ for $1 \leq j \leq 4$ in the \bmt{} with $n=4$. 
Since $\aver{i,\neg j}_{\hatB}=\aver{i,j}_{\hatB}$, we do not show the remaining columns.
The boldfaced areas correspond to cases (1) and (3) in Conjecture~\ref{conj:b-corr}.}
\label{tab:bvalues-n=4}
\end{table}

\subsection{Type \D}
\label{ss:D}

For the calculations of the limiting directions in type \D{}, we will appeal to the exclusion process known as the \dmt{} (see Section~\ref{sec:dmodel}). We will fix $n$ and let $\pi_D$ denote the stationary distribution of the \dmt{} on $n$ sites.
In this section, we denote by $\aver {i,j}$ the $\pi_D$-probability that the last two sites in the \dmt{} on $n$ sites are occupied by particles of species $i$ and $j$ respectively. Here, we take $-n \leq i,j \leq n$ and denote $-k$ by $\neg{k}$.

\begin{theorem}
\label{thm:lim-dir-d}
The limiting direction of Lam's random walk on the alcoves of the affine Weyl group $\tilde D_n$ with probabilities weighted by Kac labels is given by
\[
\sum_{i=1}^n c_i e_i,
\]
where $c_1 = 0$ if $n > 2$, and 
\[
c_i = \frac{\frac{i-1}{n+i-2}\binom{2n-3}{n-i}}{Z_{n,i}^{D}} - \frac{\frac{i-2}{n+i-3}\binom{2n-3}{n-i+1}}{Z_{n,i-1}^{D}}
\] 
for $2 \leq i \leq n$, where $Z_{n,i}^{D} = \sum_{j=0}^{n-i}\binom{2j}{j}\binom{2n-2j-2}{n-j-i}$. 
For $i = n$, $c_i$ simplifies to $\frac{1}{n}$.
\end{theorem}

The expression for the limiting direction is the same as for type \hatB{} in \eqref{eq:limB}.
The rest of the proof is very similar to the one for type $B$ and is postponed to Section \ref{sec:limD}.
Since the formula for the limiting direction is not very explicit, we give some data in Table~\ref{tab:d-ci values}.

\begin{table}[htbp!]
\[
\begin{array}{c|cccccc}
n \backslash i & 1 & 2 & 3 & 4 & 5 & 6 \\
\hline \\[-0.3cm]
2 & \frac{1}{2} & \frac{1}{2} \\[0.2cm]
3 & 0 & \frac{1}{6} & \frac{1}{3} \\[0.2cm]
4 & 0 & \frac{5}{58} & \frac{19}{116} & \frac{1}{4} \\[0.2cm]
5 & 0 & \frac{7}{130} & \frac{147}{1495} & \frac{17}{115} & \frac{1}{5}  \\[0.2cm]
6 & 0 & \frac{21}{562} & \frac{1077}{16298} & \frac{381}{3886} & \frac{53}{402} & \frac{1}{6}
\end{array}
\]
\caption{Table of values of $c_i$ from Theorem~\ref{thm:lim-dir-d} for small values of $n$.
Note that $n=2$ is a special case.}
\label{tab:d-ci values}
\end{table}

\begin{remark}
Note that the sum of positive roots for $D_n$ (from Table~\ref{tab:root-data}) is
\[
\sum_{i<j} (e_j - e_i) + \sum_{i<j} (e_i+e_j) = \sum_i (2i-2) \; e_i.
\] 
Although Lam seems to suggest that the limiting direction in type $D$ should be close to this value in~\cite[Remark 5]{lam-2015}, 
the data from Table~\ref{tab:d-ci values} suggests that this is not the case.
\end{remark}

\subsection{Type \hatC{} and \chB{}}\label{ss:hatC}

We do not consider the remaining cases of types \hatC{} and \chB{} in detail. Our techniques can be applied to these cases as well, but we do not work these out because the formulas are not as nice. 

If we write the limiting direction of Lam's random walk on the alcoves of the affine Weyl group $B_n$ (resp. $C_n$) with probabilities weighted by dual Kac labels (resp. Kac labels)
as $\sum_{i=1}^n c_i e_i$, then the values of $c_i$ are as given in Table~\ref{tab:chb-ci values} (resp. Table~\ref{tab:hatc-ci values}).

\begin{table}[htbp!]
\[
\begin{array}{c|cccc}
n \backslash i & 1 & 2 & 3 & 4 \\
\hline \\[-0.3cm]
2 & \frac{1}{10} & \frac{2}{5} \\[0.2cm]
3 & \frac{1}{22} & \frac{13}{77} & \frac{2}{7} \\[0.2cm]
4 & \frac{5}{186} & \frac{326}{3441} & \frac{52}{333} & \frac{2}{9}
\end{array}
\]
\caption{Table of values of $c_i$ for type \chB{} for small values of $n$.}
\label{tab:chb-ci values}
\end{table}

\begin{table}[htbp!]
\[
\begin{array}{c|cccc}
n \backslash i & 1 & 2 & 3 & 4 \\
\hline \\[-0.3cm]
1 & \frac{1}{2} \\[0.2cm]
2 & \frac{1}{6} & \frac{1}{3} \\[0.2cm]
3 & \frac{5}{58} & \frac{19}{116} & \frac{1}{4} \\[0.2cm]
4 & \frac{7}{130} & \frac{147}{1495} & \frac{17}{115} & \frac{1}{5}  \\[0.2cm]
\end{array}
\]
\caption{Table of values of $c_i$ for type \hatC{} for small values of $n$.}
\label{tab:hatc-ci values}
\end{table}

\begin{remark}
\label{rem:cequald}
Since the data in Tables~\ref{tab:d-ci values} and \ref{tab:hatc-ci values} are the same, it seems that the limiting directions for Lam's random walk for $D_{n+1}$ and for $C_n$ have similar formulas. Although this should be provable by plugging in $\alpha = \beta	= 1/2$ in Theorem~\ref{thm:semiperm-pf}, we do not have a conceptual understanding of this phenomenon. It would be interesting to understand this better.
\end{remark}

\section{MultiTASEP Models}
\label{sec:models}

The proofs of our theorems will follow from the analysis of certain discrete-time Markov chains known as totally asymmetric simple exclusion processes (TASEPs).  In this section, we define the TASEPs that we will need to consider.
Although TASEPs can be defined for more general graphs, it will suffice for our purposes to consider these processes on the path graph with $n$ vertices. The particle configurations satisfy the constraint that there can be at most one particle at each vertex. The dynamics is as follows. The edge between a pair of neighboring vertices is chosen with a certain probability (that depends on the specifics of the model) and the particle on the left vertex of that edge (if it exists) moves to the site on the right vertex if it is vacant. If the edge chosen is the leftmost or the rightmost, the transitions can be different, again depending on the model.

We will need to consider {\em multispecies TASEPs}, or {\em multiTASEPs} in short, in which particles of several species are present, and where there is a total order among the particles. The dynamics is then that the `larger' particle exchanges with the `smaller' particle if the former is to the left of the latter. We will label the particles $\neg n, \dots, \neg 1, 1, \dots, n$, where $n$ is the number of vertices. We think of $\neg i$ as a particle of species $-i$ so that $\neg n$ is the slowest particle and $n$ is the fastest particle moving right. In each of the multiTASEPs below, configurations consist of words of length $n$ in the alphabet $\{\neg n, \dots, \neg 1, 1, \dots, n\}$
where there is exactly one of either $i$ or $\neg i$ present for all $1 \leq i \leq n$. The sites are labeled $1,\dots, n$. In the definitions of the processes we use the Kac labels and the dual Kac labels; see Table \ref{T:Kac}. We will focus on types $\chC, \hatB$ and $\D$. The processes for types $\hatC$ and $\chB$ can be constructed in a completely analogous manner.

\subsection{\cmt{} definition}
\label{sec:cmodel}
In the \cmt{} an edge is chosen between sites $\ell$ and $\ell+1$ with probability $1/(n+1)$ for $0 \leq \ell \leq n$. If $\ell=0$ (resp. $\ell=n$), we take the edge to be the left (resp. right) boundary.
The transitions rules are given in Table~\ref{tab:cmt-transitions}.

\begin{table}[htbp!]
\resizebox{\textwidth}{!}{
\begin{tabular}{|c|c||c|c||c|c|}
\hline
\multicolumn{2}{|c||}{First site} & \multicolumn{2}{c||}{Bulk} & \multicolumn{2}{c|}{Last site} \\
Transition & Probability & Transition & Probability & Transition & Probability \\
\hline 
&&&&&\\[-0.3cm]
$\neg{k} \to k$ & $\ds \frac{1}{n+1}$ & $j i \to i j  $
& $\ds\frac{1}{n+1}$ & $k \to \neg{k}$ & $\ds \frac{1}{n+1}$  \\[0.4cm]
\hline
\end{tabular}
}
\vspace{0.5cm}
\caption{Transitions for the \cmt{}, where $\neg n \leq i < j \leq n$ and $1 \leq k \leq n$. Here the bulk includes the first and the last two sites.}
\label{tab:cmt-transitions}
\end{table}

\subsection{\bmt{} definition}
\label{sec:bmodel}

In the \bmt{}, an edge is chosen between sites $\ell$ and $\ell+1$ with probability $1/n$ for $0 \leq \ell \leq n-2$, and with probability $1/2n$ for $\ell=n-1$ or $\ell=n$. If $\ell=0$, the transition changes the sign of the first label. 
If $1\le \ell \le n-1$, the transition results in exchanging the labels at sites $\ell$ and $\ell+1$. If $\ell=n$, it exchanges the last two labels and changes their signs.
The transitions rules are given in Table~\ref{tab:bmt-transitions}.

\begin{table}[htbp!]
\renewcommand{\arraystretch}{1.2}
\resizebox{\textwidth}{!}{
\begin{tabular}{|c|c||c|c||c|c|}
\hline
\multicolumn{2}{|c||}{First site} & \multicolumn{2}{c||}{Bulk} & \multicolumn{2}{c|}{Last two sites} \\
Transition & Probability & Transition & Probability & Transition & Probability \\
\hline 
&&&&&\\[-0.3cm]
\multirow{8}{*}{$\neg{i} \to i$} & \multirow{8}{*}{$\ds \frac{1}{n}$} & 
\multirow{8}{*}{$m k \to k m  $} & 
\multirow{8}{*}{$\ds\frac{1}{n}$} & 
$j i \to \neg{i} \neg{j}$ & \multirow{8}{*}{$\ds \frac{1}{2n}$}  \\
&&&& $j i \to i j$ & \\
&&&& $j \neg{i} \to i \neg{j}$ &  \\
&&&& $j \neg{i} \to \neg{i} j$ & \\
&&&& $i j \to \neg{j} \neg{i}$ &  \\
&&&& $i \neg{j} \to \neg{j} i$ &  \\
&&&& $\neg{i} j \to \neg{j} i$ &  \\
&&&& $\neg{i} \neg{j} \to \neg{j} \neg{i}$ & \\[0.1cm]
\hline
\end{tabular}
}
\vspace{0.5cm}
\caption{Transitions for the \bmt{}, where $\neg n \leq k< m \leq n$ and $1 \leq i < j \leq n$. Here the bulk includes the first two sites.}
\label{tab:bmt-transitions}
\end{table}

\subsection{\dmt{} definition}
\label{sec:dmodel}

Configurations in this model consist of words as before with the extra condition that there are an even number of negative entries in each word. The transitions are as follows: an edge is chosen between sites $\ell$ and $\ell+1$ with probability $1/(n-1)$ for $2 \leq \ell \leq n-2$, and with probability $1/2(n-1)$ for $\ell=0, 1, n-1, n$. 
If $1\le \ell \le n-1$, the transition results in exchanging the labels at sites $\ell$ and $\ell+1$. If $\ell=0$ or $\ell=n$, it exchanges the labels of the first or last two sites, respectively, and changes their signs.
The transitions rules are given in Table~\ref{tab:dmt-transitions}.

\begin{table}[htbp!]
\renewcommand{\arraystretch}{1.2}
\resizebox{\textwidth}{!}{
\begin{tabular}{|c|c||c|c||c|c|}
\hline
\multicolumn{2}{|c||}{First two sites} & \multicolumn{2}{c||}{Bulk} & \multicolumn{2}{c|}{Last two sites} \\
Transition & Probability & Transition & Probability & Transition & Probability \\
\hline 
&&&&&\\[-0.3cm]
$\neg{i} \neg{j} \to j i$ & \multirow{8}{*}{$\ds \frac{1}{2(n-1)}$} & 
\multirow{8}{*}{$m k \to k m  $} & 
\multirow{8}{*}{$\ds\frac{1}{n-1}$} & 
$j i \to \neg{i} \neg{j}$ & \multirow{8}{*}{$\ds \frac{1}{2(n-1)}$}  \\
$\neg{i} \neg{j} \to \neg{j} \neg{i} $ &&&& $j i \to i j$ & \\
$i \neg{j} \to j \neg{i}  $ &&&& $j \neg{i} \to i \neg{j}$ &  \\
$i \neg{j} \to \neg{j} i $ &&&& $j \neg{i} \to \neg{i} j$ & \\
$\neg{j} \neg{i} \to i j  $ &&&& $i j \to \neg{j} \neg{i}$ &  \\
$j \neg{i} \to \neg{i} j $ &&&& $i \neg{j} \to \neg{j} i$ &  \\
$\neg{j} i \to \neg{i} j $ &&&& $\neg{i} j \to \neg{j} i$ &  \\
$j i \to i j  $ &&&& $\neg{i} \neg{j} \to \neg{j} \neg{i}$ & \\[0.1cm]
\hline
\end{tabular}
}
\vspace{0.5cm}
\caption{Transitions for the \dmt{}, where $\neg n \leq k < m \leq n$ and $1 \leq i < j \leq n$.}
\label{tab:dmt-transitions}
\end{table}

In a similar way one may define $\hatC$-TASEP and $\chB$-TASEP. We omit the details.

\section{Lumping to Two-Species TASEPs}
\label{sec:color}

Given a Markov chain $(X_t)$ on a state space $\Omega$ and an equivalence relation $\sim$ on $\Omega$, we naturally obtain a partition $\mathcal{S}$ of $\Omega$ into subsets. For $\omega \in \Omega$, let $[\omega]$ denote the subset containing $\omega$. One can project $(X_t)$ to a stochastic process $(Y_t)$ on 
$\mathcal{S}$ by setting $Y_t = [X_t]$.
In general, $(Y_t)$ will not be a Markov process. 
Set $\mathbb{P}_X(\sigma \to [\omega']) = \sum_{\sigma' \in [\omega']} \mathbb{P}_X(\sigma \to \sigma')$.
In the special case that $\mathbb{P}_X(\sigma \to [\omega']) = \mathbb{P}_X(\tau \to [\omega'])$ for all $\sigma,\tau \in [\omega]$, then $(Y_t)$ becomes a Markov process and we say that $Y$ is a {\em lumping} of $X$. Equivalently, we can express lumping by saying that the following diagram commutes
\begin{equation}
\label{comm-diag}
\begin{tikzcd}
\Omega \arrow[r, "M_X"] \arrow[d, "\begin{bmatrix} \cdot \end{bmatrix}"]
	& \Omega\arrow[d, "\begin{bmatrix} \cdot \end{bmatrix}"] \\
\mathcal{S}\arrow[r, "M_Y"]
	& \mathcal{S}
\end{tikzcd},
\end{equation}
where $M_X$ and $M_Y$ are the transition matrices for the two processes.
One important consequence of lumping for our purposes is the relationship between the stationary distributions of the two processes.

\begin{prop}
\label{prop:lump-ss}
Let $(X_t)$ be a Markov chain on $\Omega$ and $(Y_t)$ be a lumping of $(X_t)$ on a partition  $\mathcal{S}$ of $\Omega$. Further, let $\pi_X$ and $\pi_Y$ denote the respective stationary distributions. Then, for each $S \in \mathcal{S}$,
\[
\pi_Y(S) = \sum_{\sigma \in S} \pi_X(\sigma).
\]
\end{prop}

We will use various lumpings to compute stationary probabilities and correlations in the \cmt{}, the \bmt{} and the \dmt{}. To help the reader keep track of all the lumpings used in this article, we summarise these in Figure~\ref{fig:lumpings}.

\begin{figure}[htbp!]
\centering
\includegraphics[scale=0.8]{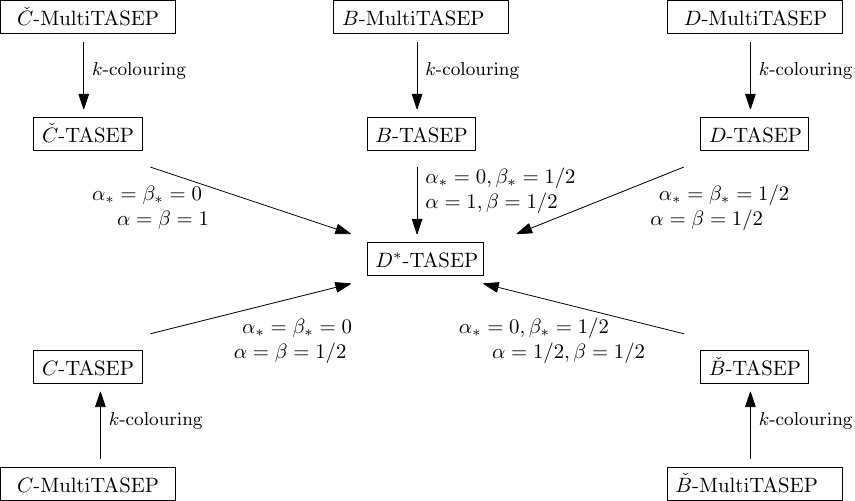}
\caption{All the TASEPs used in this article and their interrelations. All arrows correspond to lumpings. For completeness, we have also included the $\hatC$-MultiTASEP and $\chB$-MultiTASEP although we do not discuss the details in this article.}
\label{fig:lumpings}
\end{figure}

\subsection{Two-species TASEPs}

In this section, we will define two-species TASEPs on a finite one-dimensional lattice of size $n$, where the number of vacancies (i.e. $0$'s) is fixed to be $n_0$. Let
\[
\Omega_{n,n_0} = \{ \tau \in \{\1,0,1\}^n \mid \text{the number of 0's in $\tau$ is $n_0$} \}.
\]

We give the details for \ct{}, \bt{} and \dt{} whose transition probabilities are given in Tables~\ref{tab:ct-transitions}, \ref{tab:bt-transitions} and \ref{tab:dt-transitions}. We omit the details for $\hatC$-TASEP and $\chB$-TASEP.

\begin{table}[htbp!]
\resizebox{\textwidth}{!}{
\begin{tabular}{|c|c||c|c||c|c|}
\hline
\multicolumn{2}{|c||}{First site} & \multicolumn{2}{c||}{Bulk} & \multicolumn{2}{c|}{Last site} \\
Transition & Probability & Transition & Probability & Transition & Probability \\
\hline 
&&&&&\\[-0.3cm]
\multirow{3}{*}{$\1 \to 1$} & 
\multirow{3}{*}{$\ds \frac{1}{n+1}$} & 
$1 \1 \to \1 1  $ & \multirow{3}{*}{$\ds\frac{1}{n+1}$} & 
\multirow{3}{*}{$1 \to \1$} & 
\multirow{3}{*}{$\ds \frac{1}{n+1}$} \\
&& $1 0 \to 0 1  $ & & & \\
&& $0 \1 \to \1 0  $ & &  &  \\[0.1cm]
\hline
\end{tabular}
}
\vspace{0.5cm}
\caption{Transitions for the \ct{}.}
\label{tab:ct-transitions}
\end{table}

\begin{table}[htbp!]
\resizebox{\textwidth}{!}{
\begin{tabular}{|c|c||c|c||c|c|}
\hline
\multicolumn{2}{|c||}{First site} & \multicolumn{2}{c||}{Bulk} & \multicolumn{2}{c|}{Last two sites} \\
Transition & Probability & Transition & Probability & Transition & Probability \\
\hline 
&&&&&\\[-0.3cm]
\multirow{6}{*}{$\1 \to 1$} & 
\multirow{6}{*}{$\ds \frac{1}{n}$} & && $1 1 \to \1 \1$ & \multirow{6}{*}{$\ds \frac{1}{2n}$}  \\
&&  & & $1 \1 \to \1 1$ & \\
&& $1 \1 \to \1 1  $ & \multirow{3}{*}{$\ds\frac{1}{n}$} & $0 1 \to \1 0$ &  \\
&& $1 0 \to 0 1  $ & & $0 \1 \to \1 0$ & \\
&& $0 \1 \to \1 0 $ && $1 0 \to 0 \1$ &  \\
&&&& $1 0 \to 0 1$ &  \\[0.1cm]
\hline
\end{tabular}
}
\vspace{0.5cm}
\caption{Transitions for the \bt{}.}
\label{tab:bt-transitions}
\end{table}

\begin{table}[htbp!]
\renewcommand{\arraystretch}{1.2}
\resizebox{\textwidth}{!}{
\begin{tabular}{|c|c||c|c||c|c|}
\hline
\multicolumn{2}{|c||}{First two sites} & \multicolumn{2}{c||}{Bulk} & \multicolumn{2}{c|}{Last two sites} \\
Transition & Probability & Transition & Probability & Transition & Probability \\
\hline 
&&&&&\\[-0.3cm]
$\1 \1 \to 1 1$ & \multirow{6}{*}{$\ds \frac{1}{2(n-1)}$} & && $1 1 \to \1 \1$ & \multirow{6}{*}{$\ds \frac{1}{2(n-1)}$}  \\
$1 \1 \to \1 1$ && &  & $1 \1 \to \1 1$ & \\
$\1 0 \to 0 1  $ && $1 \1 \to \1 1  $ & \multirow{3}{*}{$\ds\frac{1}{n-1}$} & $0 1 \to \1 0$ &  \\
$1 0 \to 0 1 $ && $1 0 \to 0 1  $ && $0 \1 \to \1 0$ & \\
$0 \1 \to 1 0  $ && $0 \1 \to \1 0  $ && $1 0 \to 0 \1$ &  \\
$0 \1 \to \1 0 $ &&&& $1 0 \to 0 1$ &  \\[0.1cm]
\hline
\end{tabular}
}
\vspace{0.5cm}
\caption{Transitions for the \dt{}.}
\label{tab:dt-transitions}
\end{table}

We observe one symmetry property for the \ct{} and the \dt{} which will prove useful later.
The proof is an easy exercise and can be proved by a case analysis.

\begin{prop}
\label{prop:ctdt-symm}
The \ct{} and the \dt{} on $\Omega_{n,n_0}$ are invariant as Markov chains under the transformation
$\tau = (\tau_1,\dots,\tau_n) \to (-\tau_n, \dots, -\tau_1)$, where $-1$ and $- \1$ are to be interpreted as $\1$ and $1$ respectively.
\end{prop}

We will now define a special class of lumpings for multiTASEPs. 
For each $X$-MultiTASEP on $n$ sites, where $X \in \{\hatB, \chB, \hatC, \chC, D \}$, we will define a lumping to the $X$-TASEP as follows. 
Fix $1 \leq k \leq n$. The {\em $k$-coloring} is a map $f_k: \Omega^{\text{X}}_n$ to $\Omega_{n,k}$ defined as follows:
\begin{equation}
f_k(\tau_1,\dots,\tau_n) = (f_k(\tau_1), \dots, f_k(\tau_n)),
\end{equation}
where
\[
f_k (i) = \begin{cases}
1 & i \geq k, \\
\1 & i \leq -k, \\
0 & -k < i < k.
\end{cases}
\]
The reason this map is called $k$-coloring is that if one imagines all particles to be of different colors, then this is the TASEP imagined through the eyes of a colorblind person, whose blindness is one that distinguishes only between `higher', `lower' and `medium' colors. 

\begin{prop}
\label{prop:kcoloring}
For $X \in \{\hatB, \chB, \hatC, \chC, D\}$, the $k$-coloring is a lumping from the $X$-MultiTASEP with $n$ sites to the $X$-TASEP with $n$ sites and $k-1$ $0$'s.
\end{prop}

\begin{proof}
We will check the commutativity \eqref{comm-diag} of the $k$-coloring for all possible transitions. In principle, this will have to be done on a case-by-case basis depending on the particles hopping. However, one can group many of these cases together. 

We illustrate this for the bulk transitions in the \cmt{}. The argument is similar for the other types.
Compare the central columns of Tables~\ref{tab:cmt-transitions} and \ref{tab:ct-transitions}.
If both $i,j > k$, then there is no transition in the \ct{}. Similarly, if $-k \leq i,j \leq k$ or $i,j < -k$. If $j > k \geq i \geq -k$, then the transition becomes $10 \to 01$. In both cases, the probability is $1/(n+1)$.
Similarly, one can check all other cases. 
For the boundary transitions, this is again a case analysis, and is left to the reader.
\end{proof}

\subsection{The \dstar{}}

We will prove results for the classical types by appealing to a two-species exclusion process known as the \dstar{} 
studied in~\cite{AALP-2019}. The \dstar{} is defined on a finite one-dimensional lattice of $n$ sites. On each site, we have exactly one particle from the set $\{*,1,0,\1\}$ subject to the following conditions:
\begin{itemize}
\item the number of $0$'s is fixed to be $n_0$;
\item sites $1$ and $n$ can only be occupied by $0$ and $*$;
\item sites $2$ through $n-1$ can only be occupied by $1,0$ and $\1$.
\end{itemize}
Let $\Omega^*_{n,n_0}$ denote the possible configurations. For example, 
\[
\Omega^*_{3,1} = \{ (0, \1, *), (0, 1, *), (*, \1, 0), (*, 1, 0), (*, 0, *) \}.
\]
Let $\alpha, \alpha_*, \beta, \beta_* \in [0,1].$ 

For our purposes, it will be convenient to think of the \dstar{} as a (discrete-time) Markov chain. Although this is defined as a continuous-time Markov process in~\cite{AALP-2019}, the stationary distributions in both cases are identical and that is what is relevant for us here.

The transitions are as follows.
With probability $1/(n-1)$, the edge between sites $\ell$ and $\ell+1$ is chosen, where $\ell\in [n-1]$. If $2 \leq \ell \leq n-2$, then a transition occurs interchanging the particles at sites $\ell$ and $\ell+1$ if 
the particle at site $\ell$ is larger than that at $\ell+1$, where we interpret $\1$ as $-1$. If $\ell=1$ or $\ell=n-1$ the probability of a transition is multiplied with the parameters according to Table \ref{tab:dstar-transitions}.
With the remaining probability, the configuration remains unchanged.

\begin{table}[htbp!]
\renewcommand{\arraystretch}{1.1}
\resizebox{\textwidth}{!}{
\begin{tabular}{|c|c||c|c||c|c|}
\hline
\multicolumn{2}{|c||}{First two sites} & \multicolumn{2}{c||}{Bulk} & \multicolumn{2}{c|}{Last two sites} \\
Transition & Probability & Transition & Probability & Transition & Probability \\
\hline 
&&&&&\\[-0.3cm]
$* \1 \to * 1$ & $\ds \frac{\alpha}{n-1}$ & $1 \1 \to \1 1  $ &  \multirow{3}{*}{$\ds\frac{1}{n-1}$} & $1 * \to \1 *$ & $\ds \frac{\beta}{n-1}$  \\[0.4cm]
$* 0 \to 0 1$ & $\ds \frac{\alpha_*}{n-1}$ & $1 0 \to 0 1  $ & & $0 * \to \1 0$ & $\ds \frac{\beta_*}{n-1}$  \\[0.4cm]
$0 \1 \to * 0$ & $\ds \frac{1}{n-1}$ & $0 \1 \to \1 0  $ & & $1 0 \to 0 *$ & $\ds \frac{1}{n-1}$  \\[0.2cm]
\hline
\end{tabular}
}
\vspace{0.5cm}
\caption{Transitions for the \dstar{}.}
\label{tab:dstar-transitions}
\end{table}

The stationary distribution of the \dstar{}, denoted $\pi^*$, was obtained using the technology of the matrix ansatz in~\cite{AALP-2019}. Here, it will be more convenient for us to obtain the stationary distribution using a process which lumps to the \dstar{}. This is done in Section~\ref{sec:duchi-schaeffer}.

As stated already in \cite{AALP-2019}, one can check that if $\alpha_*$ or $\beta_*$ are zero, then the \dstar{} is not ergodic. If the former, there are no outgoing transitions from states which begin with a $*$. Similarly for the latter. If $\alpha_* = \beta_* = 0$, there are no outgoing transitions from states which both begin and end with a $*$. 
The following result is then easy to see.

\begin{prop}
\label{prop:dstar-irred}

Suppose $\alpha, \beta > 0$ and $\alpha_*, \beta_* \geq 0$.

\begin{enumerate}
\item If both $\alpha_*$ or $\beta_*$ are nonzero,  the \dstar{} is irreducible.

\item If $\alpha_* = 0$ and $\beta_* \neq 0$, the \dstar{} restricted to configurations which have a $*$ at the first site is irreducible.

\item If $\alpha_* = \beta_* = 0$, the \dstar{} restricted to configurations which have a $*$ at both the first and last sites is irreducible.

\end{enumerate}
\end{prop}

We are now in a position to prove the main result of this section.

\begin{theorem} \leavevmode
\label{thm:lump}
\begin{enumerate}

\item In the \dstar{} with $n+2$ sites and $n_0$ $0$'s and $\alpha_* = \beta_* = 0, \alpha = \beta = 1$, the marginal process of sites $2$ through $n+1$ is isomorphic to the \ct{} on $n$ sites with $n_0$ $0$'s.

\item The \bt{} with $n$ sites and $n_0$ $0$'s lumps to the marginal process of sites $2$ through $n+1$ of the \dstar{} on $n+1$ sites with $n_0$ $0$'s and $\alpha_* = 0, \alpha = 1, \beta = \beta_* = 1/2$.

\item The \dt{} with $n$ sites and $n_0$ $0$'s lumps to the \dstar{} on $n$ sites with $n_0$ $0$'s and $\alpha = \alpha_* = \beta = \beta_* = 1/2$.

\end{enumerate}
\end{theorem}

\begin{proof}
From Proposition~\ref{prop:dstar-irred}(3) and Tables~\ref{tab:ct-transitions} and \ref{tab:dstar-transitions}, the proof of (1) immediately follows.

The idea for the proofs of parts (2) and (3) is another kind of coloring argument, namely for particles $1$ and $\1$ can be identified at the endpoints. For the \bt{} with $n$ sites and $n_0$ $0$'s, this identification is done at site $n$ and the resulting `particle' is labelled $*$. For the \dt{} with $n$ sites and $n_0$ $0$'s, this identification is done for both sites $1$ and $n$. We will give the idea of the proof for the \bt{}. Similar ideas hold for the \dt{}.

Using Proposition~\ref{prop:dstar-irred}(2) and comparing Tables~\ref{tab:bt-transitions} and \ref{tab:dstar-transitions}, it is clear that the transitions at the first site and the bulk are unaffected, just as
for \ct{} the first site of \dstar{} stays a $*$. 
We only need to compare transitions at the last two sites. For the reader's convenience, we reproduce the transition rates here after adjusting the value of $n$ and setting $\beta = \beta_* = 1/2$.
\[
\begin{array}{|c|c|c|c|}
\hline
\multicolumn{2}{|c|}{\text{\bt{}}} & \multicolumn{2}{c|}{\text{\dstar{}}} \\
\text{Transition} & \text{Probability} & \text{Transition} & \text{Probability} \\
\hline
1 1 \to \1 \1 & \multirow{6}{*}{$\ds \frac{1}{2n}$}  & \multirow{2}{*}{$1 * \to \1 * $} & \multirow{2}{*}{$\ds \frac{1}{2n} $}  \\
1 \1 \to \1 1 & & &\\[0.1cm]
0 1 \to \1 0 &  & \multirow{2}{*}{$0 * \to \1 0$}  & \multirow{2}{*}{$\ds \frac{1}{2n}$}   \\
0 \1 \to \1 0 & & &\\[0.1cm]
1 0 \to 0 \1 &  & \multirow{2}{*}{$1 0 \to 0 *$}  & \multirow{2}{*}{$\ds \frac{1}{n} $}  \\
1 0 \to 0 1 &  & & \\[0.1cm]
\hline
\end{array}
\]
By identifying $1$ and $\1$ to $*$ in the last site, the six transitions in the first column become the transitions in the third column. It is also easy to see that the rates match correctly. In particular, the rates of the last pair of transitions on the left add up to give the last rate on the right. 
This completes the proof.
\end{proof}

As a consequence of Proposition~\ref{prop:kcoloring} and Theorem~\ref{thm:lump}, it is enough to study correlations in the \dstar{} to determine the necessary correlations in the various multiTASEPs.
This is the major technical part of this work and is taken up in Section~\ref{sec:part}.
First, we study a two-row process in Section \ref{sec:duchi-schaeffer}.

\section{The Two-Row $D^*$ Process}
\label{sec:duchi-schaeffer}

In this section, we define a Markov chain we call the {\em two-row $D^*$ process} and we will prove that it lumps to the \dstar{}.
The strategy is similar to the one used by Duchi and Schaeffer~\cite{duchi-schaeffer-2005} for what they call the 3-TASEP.

The configurations for the two-row $D^*$ process are as follows.
A \emph{two-row configuration} is a pair of rows of $n$ sites, each of which contains exactly one of  $\neg{1}$, $0$, $1$ or $*$, satisfying the following conditions:
\begin{itemize}
\item A $0$ or a $*$ occurs in the top row of a column if and only if it occurs in the bottom row of the same column. A column
containing $0$'s or $*$'s is called a {\em $0$-column} or a {\em $*$-column}, respectively.

\item Leftmost and rightmost columns can only be $0$- or $*$-columns. In addition, $*$-columns cannot appear elsewhere.

\item The {\em balance condition}: there is an equal amount of 1's and $\neg{1}$'s between any 0-columns.

\item The {\em positivity condition}: there are at least as many 1's as $\neg{1}$'s to the left of any column.

\end{itemize}

Let $\widehat{\Omega}^*_{n, n_0}$ be the set of two-row configurations with $n$ columns and $n_0$ 0-columns. 
For example,
\[
\widehat{\Omega}^*_{3, 1} = \left\{
\xx \bw \st, \; \xx \wb \st, \; \st \xx \st, \; \st \bw \xx, \; \st \wb \xx
\right\}
\]
and
\[
\widehat{\Omega}^*_{4, 0} = \left\{
\st \wb \wb \st, \; \st \wb \bw \st, \; \st \bb \ww \st, \; \st \bw \wb \st, \; \st \bw \bw \st
\right\}.
\]
For $\omega \in \widehat{\Omega}^*_{n, n_0}$, we say that the {\em wall $i$} is the vertical line between columns $i$ and $i+1$ for $1 \leq i \leq n-1$, and we denote by $\omega[i]$ the four sites around the wall $i$. 
Let $j_1 < i$ be the leftmost wall such that there are only $\neg{1}$'s on the top row between it and the wall $i-1$.
In particular, if there is no $\neg{1}$'s on the top row at site $i-1$, then $j_1 = i-1$.
Similarly, let $j_2 > i$ be the rightmost wall such that there are only $1$'s between the walls $i+1$ and $j_2$ on the top row.
To define the two-row $D^*$ process, we will first need a map $T_* : \widehat{\Omega}^*_{n, n_0} \times [n-1] \rightarrow \widehat{\Omega}^*_{n, n_0}$ given as follows. Let $\omega \in \widehat{\Omega}^*_{n, n_0}$ and $i \in [n-1]$. We describe $\omega' = T_*(\omega,i)$ now. There are three cases to consider:

\begin{itemize}
\item $2 \leq i \leq n - 2$, called a {\em bulk transition}: 

\begin{enumerate}[label=(B\arabic*)]

\item \label{it:B1}
If $\omega[i] =$ \bq$|$\wb~or \xx$|$\wb, then 
$\omega'$ is obtained by moving the $\wb$ from the right-hand side of $i$ to the right of the wall $j_1$. See Figure \ref{fig:b1_example} for an illustration.

\begin{figure}[htbp!]
\[
\scalebox{1.95}{${ \atop \bq}{\scalebox{0.5}{$j_1$} \atop |}{ \atop \wq}{ \atop {\dots \atop \dots}}{ \atop \wq}{ \atop \bq}{\scalebox{0.5}{$i$} \atop \textbf{\textbar}}{ \atop \wb}
{ \atop \rightarrow}
{ \atop \bq}{\scalebox{0.5}{$j_1$} \atop |}{ \atop \wb}{ \atop \wq}{ \atop {\dots \atop \dots}}{ \atop \wq}{\scalebox{0.5}{$i$} \atop |}{ \atop \bq}$ }
\]
\caption{A bulk transition of type~\ref{it:B1}.}
\label{fig:b1_example}
\end{figure}

\item \label{it:B2}
If $\omega[i] =$ \bq$|$\ww~or \bw$|$\xx, then 
$\omega'$ is obtained by removing the two particles that form the \dg~or the $\bw$ at $i$ and placing them at $j_2$ so that they form a \dg~if there is a $\neg{1}$ on the right-hand side of $j_2$ in the top row, or otherwise form a $\bw$ on the left-hand side of $j_2$. An illustration is provided in Figure \ref{fig:b2_example}.

\begin{figure}[htbp!]
\[\scalebox{1.95}{${ \atop \bq}{\scalebox{0.5}{$i$} \atop \textbf{\textbar}}{ \atop \ww}{ \atop \bq}{ \atop {\dots \atop \dots}}{ \atop \bq}{\scalebox{0.5}{$j_2$} \atop |}{ \atop \wq}
{ \atop \rightarrow}
{ \atop \wq}{\scalebox{0.5}{$i$} \atop |}{ \atop \bq}{ \atop {\dots \atop \dots}}{ \atop \bq}{ \atop \bq}{\scalebox{0.5}{$j_2$} \atop |}{ \atop \ww}$ }\]
\caption{A bulk transition of type~\ref{it:B2}.}
\label{fig:b2_example}
\end{figure}

\end{enumerate}

\item $i = 1$, called a {\em left border transition}:

\begin{enumerate}[label=(L\arabic*)]

\item \label{it:L1}
If $\omega[1] = \st|\wb$, we ignore the first site 
and obtain $\omega'$ by removing the $\wb$ on the left border and placing the particles at $j_2$ so that they form a \dg~if there is a $\neg{1}$ on the right-hand side of $j_2$ in the top row, or otherwise form a $\bw$ on the left-hand side of $j_2$. See Figure \ref{fig:l1_example} for an illustration.

\begin{figure}[htbp!]
\[
\scalebox{1.95}{${ \atop \st}{\scalebox{0.5}{$i = 1$} \atop \textbf{\textbar}}{ \atop \wb}{ \atop \bq}{ \atop {\dots \atop \dots}}{ \atop \bq}{\scalebox{0.5}{$j_2$} \atop |}{ \atop \xx}
{ \atop \rightarrow}
{ \atop \st}{{\scalebox{0.5}{$i = 1$} \atop |}}{ \atop \bq}{ \atop {\dots \atop \dots}}{ \atop \bq}{ \atop \bw}{\scalebox{0.5}{$j_2$} \atop |}{ \atop \xx}$ }
\]
\[
\scalebox{1.95}{${ \atop \st}{\scalebox{0.5}{$i = 1$} \atop \textbf{\textbar}}{ \atop \wb}{ \atop \bq}{ \atop {\dots \atop \dots}}{ \atop \bq}{\scalebox{0.5}{$j_2$} \atop |}{ \atop \wq}
{ \atop \rightarrow}
{ \atop \st}{{\scalebox{0.5}{$i = 1$} \atop |}}{ \atop \bq}{ \atop {\dots \atop \dots}}{ \atop \bq}{ \atop \bq}{\scalebox{0.5}{$j_2$} \atop |}{ \atop \ww}$ }
\]
\caption{Transitions of type~\ref{it:L1} at the left border.}
\label{fig:l1_example}
\end{figure}

\item \label{it:L2}
If $\omega[1] = \st | \xx$, we pretend $\st$ is $\bw$ and perform bulk transition~\ref{it:B2}.

\item \label{it:L3}
If $\omega[1] = \xx | \wb$, then 
we make a transition to $\omega'$, where the only change is that $\omega'[1]$ becomes $\st | \xx$. 

\end{enumerate}

\item $i = n-1$, called a {\em right border transition}:

\begin{enumerate}[label=(R\arabic*)]
\item \label{it:R1}
If $\omega[n-1] = \bw | \st$, we ignore the last site and 
obtain $\omega'$ by removing the rightmost $\bw$ and forming a $\wb$ on the right-hand side of wall $j_1$. See Figure \ref{fig:r1_example} for an illustration.

\begin{figure}[htbp!]
\[
\scalebox{1.95}{${ \atop \xx}{\scalebox{0.5}{$j_1$} \atop |}{ \atop \wq}{ \atop {\dots \atop \dots}}{ \atop \wq}{ \atop \bw}{\scalebox{0.35}{$i = n-1$} \atop \textbf{\textbar}}{ \atop \st}
{ \atop \rightarrow}
{ \atop \xx}{\scalebox{0.5}{$j_1$} \atop |}{ \atop \wb}{ \atop \wq}{ \atop {\dots \atop \dots}}{ \atop \wq}{\scalebox{0.35}{$i = n-1$} \atop |}{ \atop \st}$ }
\]
\caption{A transition of type~\ref{it:R1} at the right border.}
\label{fig:r1_example}
\end{figure}

\item \label{it:R2}
If $\omega[n-1] = \xx | \st$, we pretend $\st$ is $\wb$ and perform bulk transition~\ref{it:B1}.

\item \label{it:R3} 
If $\omega[n-1] = \bw | \xx$, then 
we make a transition to $\omega'$, where the only change is that $\omega'[n-1] = \xx | \st$.

\end{enumerate}

\end{itemize}

In all other cases, $T_*(\omega,i) = \omega$.
Let $\alpha, \alpha_*, \beta, \beta_* \in [0,1].$ 
The two-row $D^*$ process on $\widehat{\Omega}^*_{n, n_0} $ is then given by first picking a wall $i$ uniformly at random among $[n-1]$ 
and then making the transition to $T_*(\omega,i)$ with probability $\lambda(\omega[i])$ given by the following table:
\begin{equation}
\label{tworow-rates}
\begin{array}{|c|c|c|c|c|c|c|c|c|}
\hline
\text{Transition} &  \ref{it:B1} & \ref{it:B2} & \ref{it:L1} & \ref{it:L2} & \ref{it:L3} & 
\ref{it:R1} & \ref{it:R2} & \ref{it:R3}\\
\hline
\lambda(\omega[i]) & 1 & 1 & \alpha & \alpha_* &  1 & \beta & \beta_* & 1 \\
\hline
\end{array}
\end{equation}

\begin{prop}
\label{prop:tworow-irred}
If $\alpha, \alpha_*, \beta, \beta_* \in (0,1]$, the two-row $D^*$ process described above is irreducible and aperiodic.
\end{prop}

\begin{proof}
Restricting to columns $2, \dots, n-1$ and transitions~\ref{it:B1},~\ref{it:B2}, \ref{it:L1} and~\ref{it:R1}, the process coincides with the Duchi-Schaeffer two-row process with a fixed number of neutral particles; see \cite[Section 4]{duchi-schaeffer-2005}. They show in \cite[Section 6]{duchi-schaeffer-2005} that this process is irreducible. Since~\ref{it:L3} changes a 0-column into a *-column and~\ref{it:L2} vice versa at the left border, and~\ref{it:R3} and~\ref{it:R2} are their counterparts at the right border, the two-row $D^*$ process is irreducible. Since there are several transitions for which $\omega' = \omega$, the process is clearly aperiodic.
\end{proof}

Before we compute the stationary distribution of the two-row $D^*$ process, we state an important property justifying the usefulness of this process. Recall the definition of lumping from Section~\ref{sec:color}.

\begin{prop}
\label{prop:tworow-lump}
The two-row $D^*$ process lumps to the \dstar{}.
\end{prop}

\begin{proof}
By comparing with Table~\ref{tab:dstar-transitions}, one can see that the transitions in the top row of the two-row $D^*$ process are identical to those of the \dstar{}. In particular, transitions~\ref{it:L1}, \ref{it:L2} and \ref{it:L3} match those in the left columns, 
transitions~\ref{it:B1} and \ref{it:B2} match those in the middle columns,
and transitions~\ref{it:R1}, \ref{it:R2} and \ref{it:R3} match those in the right columns. The fact that the probabilities are the same is obtained by comparing  Table~\ref{tab:dstar-transitions} with \eqref{tworow-rates}.
\end{proof}

We now extend the map $T_*$ to $\bar{T_*}: \widehat{\Omega}^*_{n, n_0} \times [n-1] \rightarrow \widehat{\Omega}^*_{n, n_0} \times [n-1]$
where, if $(\omega', j) = \bar{T_*}(\omega, i)$, then $\omega' = T_*(\omega, i)$ and the value of $j$ depends on the transition according to the following rules:
\begin{equation}
\label{tworow-inverse}
\begin{array}{|c|c|c|c|c|c|c|c|c|}
\hline
\text{Transition} &  \ref{it:B1} & \ref{it:B2} & \ref{it:L1} & \ref{it:L2} & \ref{it:L3} & 
\ref{it:R1} & \ref{it:R2} & \ref{it:R3}\\
\hline
\text{Value of } j & j_1 & j_2 & j_2 & j_2 &  1 & j_1 & j_1 & n-1 \\
\hline
\end{array}
\end{equation}
If $T_*(\omega, i) = \omega$, we define $j = i$.

\begin{prop}
\label{prop:tworow-bij}
The map $\bar{T_*}$ is a bijection.
\end{prop}

\begin{proof}

To see why $\bar{T}_*$ is a bijection from $\widehat{\Omega}^*_{n, n_0} \times [n-1]$ to itself, consider the 34 possible local configurations $\omega'[j]$. It is straightforward to check that 24 of these satisfy $\bar{T}_*(\omega', i) = (\omega', i)$. We then have essentially four different cases left for which we find the pre-images $(\omega, i)$.

\begin{itemize}

\item If $\omega'[j] = \st|\xx$ (so $j = 1$), the transition performed must have been~\ref{it:L3}, so $i = 1$ and $\omega[i] = \xx|\wb$. The remaining columns of $\omega$ and $\omega'$ are the same, so whenever $\omega'$ is a valid configuration, the pre-image $(\omega, i)$ has to exist.

\item If $\omega'[j] = \xx|\st$ (so $j = n-1$), the transition performed must have been~\ref{it:R3}, so $i = n-1$ and $\omega[i] = \bw|\xx$. The remaining columns of $\omega$ and $\omega'$ are the same, so whenever $\omega'$ is a valid configuration, the pre-image $(\omega, i)$ must exist.

\item If $\omega'[j] = \st|\wb, \xx|\wb, \bb|\wb$ or $\bw|\wb$, the transition performed must have been~\ref{it:B1},~\ref{it:R1} or~\ref{it:R2}. To find the pre-image $(\omega, i)$, we remove the $\wb$ and move to the right until encountering a 1, * or 0 on the top row in some column $c$. In the first two cases the $\wb$ is inserted to the left of column $c$. In the third, we insert the $\wb$ to the right of column $c$, but if it is the rightmost column, it becomes a *-column. In all cases, $i$ is the wall between the inserted column and $c$. Note that if $\omega'$ is valid, no conditions can become violated in $\omega$.

\item If $\omega'[j] = \bw|\st, \bw|\xx, \bb|\ww$ or $\bw|\ww$, the transition performed must have been~\ref{it:B2},~\ref{it:L1} or~\ref{it:L2}. To find the pre-image $(\omega, i)$, in the first two cases we remove the $\bw$ and in the latter two the $\dg$, and then move to the left until encountering a $\1$, * or 0 in the top row in some column $c$. If a $\wq$ is encountered first, we insert a $\dg$ to get $\wq | \ww = \omega[i]$. In the *-case we place $\wb$ to the right of column $c$. Finally, in the 0-case we insert $\bw$ to the left of column $c$, but if it is the leftmost column, it becomes a *-column. In both cases $i$ is the wall between column $c$ and the inserted column. Again, if $\omega'$ is valid, no conditions can become violated in $\omega$. Note, in particular, that changing $\wq$ to $\bq \ww$ preserves the balance and positivity conditions.

\end{itemize}

This completes the proof.
\end{proof}

A \emph{block} is a part of a two-row configuration of the form $|$\bb$|\omega'|$\ww$|$, where $\omega'$ is called the \emph{inside} of the block. For the purpose of describing the stationary distribution, we introduce the following labelling for $\omega \in \widehat{\Omega}^*_{n, n_0}$: 
\begin{enumerate}

\item[$z$.] Label each $\1$ on the bottom row, to the right of the rightmost 0 (if $n_0 > 0$), and not in a block by a $z$.

\item[$z'$.] Label each $\1$ on the bottom row, to the left of the leftmost 0 (if $n_0 > 0$), and not in a block by a $z'$.

\item[$y$.] Label each 1 on the bottom row, to left of the leftmost 0 (if $n_0 > 0$), not inside a block, and such that there is no $z'$ to the left by a $y$.

\end{enumerate}

An example is provided in Figure \ref{fig:label_example}. Note that for $n_0 = 0$, a $\1$ may be labelled with both $z$ and $z'$ simultaneously. Now, we define $n_{y}(\omega)$ to be the number of $y$-labels, $n_z(\omega)$ the number of $z$-labels,
\[
n_{y_*}(\omega) = 
\begin{cases} 
1, \textrm{if there is a *-column at the left border} \\ 
0, \textrm{otherwise},
\end{cases}
\]
and 
\[
n_{z_*}(\omega) = 
\begin{cases} 
1, \textrm{if there is a *-column at the right border} \\ 
0, \textrm{otherwise}.
\end{cases}
\] 
Let 
\begin{equation}
\label{def-q}
q(\omega) = \frac{1}{\alpha^{n_{y}(\omega)} \alpha_*^{n_{y_*}(\omega)}  \beta^{n_{z}(\omega)} \beta_*^{n_{z_*}(\omega)} }.
\end{equation}
For example, $q(\omega) = \frac{1}{\alpha \beta \beta_*}$ for the configuration in Figure \ref{fig:label_example}.

\begin{figure}[htbp!]
\[
\scalebox{2}{${\bb \atop }{\wb \atop }{\ww \atop }{\wb\atop \scriptscriptstyle y}{\bw \atop \scriptscriptstyle z'}{\wb \atop }{\xx \atop }{\bb \atop }{\ww \atop }{\bw \atop\scriptscriptstyle z}{\st \atop }$}
\]
\caption{A two-row configuration and its labelling.}
\label{fig:label_example}
\end{figure}

We now state the main result of this section. By Proposition~\ref{prop:tworow-irred}, the two-row $D^*$ process has a unique stationary distribution, which we will denote by $\hat{\pi}^*$. 

\begin{theorem}
\label{thm:dist_d}
The stationary distribution of the two-row $D^*$ process is given by
\[
\hat{\pi}^*(\omega) = \frac{q(\omega)}{Z_{n,n_0}^*},
\]
where 
\[
Z_{n,n_0}^* = \sum_{\omega \in \widehat{\Omega}^*_{n, n_0}} q(\omega).
\]
\end{theorem}

\begin{remark}
When $\alpha_*=0$, the transition~\ref{it:L2} does not occur. Therefore, the two-row $D^*$ process is irreducible only on
configurations which begin with a $*$-column, for which $n_{y_*}(\omega) = 1$. 
The way we interpret the stationary weights $q(\omega)$ from \eqref{def-q} is that we simply ignore the factor proportional to $\alpha_*$. 
Similar remarks apply to the case when $\beta_*=0$. When both $\alpha_*$ and $\beta_*$ are zero, we ignore both these factors in $q(\omega)$.
\end{remark}

The lemma below is key in the proof of Theorem \ref{thm:dist_d}.

\begin{lemma}
\label{lem:transfer_d}
For $(\omega, i) \in \widehat{\Omega}^*_{n, n_0} \times [n-1]$,
$$\lambda(\omega[i])q(\omega) = \lambda(\omega'[j])q(\omega'),$$ where $(\omega', j) = \bar{T_*}(\omega, i)$.
\end{lemma}

\begin{proof}
The proof is case-by-case. We begin with an observation.
In \bq$|$\wb, the bottom 1 does not have a label. If $? = 1$, it is inside a block, and if $? = \1$, there either is a $z'$ or a 0-column to the left, or it is inside a block.

The cases below are labelled according to the definition of the transitions. 
\begin{itemize}

\item Bulk:

\begin{enumerate}

\item[(B1)] Changing the place of the column $\wb$ does not affect the labels of other particles since it does not alter blocks and cannot have the label $z'$. We have two initial cases, $\bq|\wb$ and $\xx|\wb$. The possible local configurations around $j = j_1$ after a transition are \bq$|$\wb~($\lambda(\omega'[j]) = 1$), \xx$|$\wb~$(\lambda(\omega'[j]) = 1)$ and \st$|$\wb~($\lambda(\omega'[j]) = \alpha$). In the first case, the moved column $\wb$ remains unlabelled by observation 1. In the second case, $\wb$ cannot pass the leftmost 0-column and hence remains unlabelled. Finally, in the third case, we introduce a label $y$ either by observation 1 or since the moved column passes the leftmost 0-column. 

\item[(B2)] We have two initial cases, $\bq|\ww$ and $\bw|\xx$. In the latter, moving the $\bw$ has no effect on other particles. In the former, the bottom $\1$ is in a block and thus has no label. Moving the $\dg$ leaves either $\ww$, which is in a block and has no label, or $\wb$ if the initial configuration was $\bb | \ww$. No other particles are affected. The label of the remaining bottom 1 is preserved. The possible local configurations around $j = j_2$ after a transition are \bq$|$\ww~($\lambda(\omega'[j]) = 1$), \bw$|$\xx~$(\lambda(\omega'[j]) = 1)$ and \bw$|$\st~($\lambda(\omega'[j]) = \beta$). In the first case, the bottom $\1$ has no label since it is in a block. The same is true in the second since there is a 0-column to the right. In the third, the $\1$ gets the label $z$. No other particles are affected in these cases. 

\end{enumerate}

\item Left border:

\begin{enumerate}

\item[(L1)] The removal of \wb~removes a label $y$. $\lambda(\omega[i]) = \alpha.$ There are three cases: either there is a $\1$ to the right of $j$ and we insert \dg, there is a 0-column to the right and we insert \bw~to its left, or there is a *-column to the right and we insert \bw~to its left. In the latter two the insertion of \bw~clearly does not have an effect on other labels. In the 0-column case the $\1$ is labelled by $z'$. Hence $\lambda(\omega[i]) q(\omega) = \lambda(\omega'[j])q(\omega')$, since $\lambda(\omega'[j]) = 1$. In the *-case, $\lambda(\omega'[j]) = \beta$ and we introduce the label $z$. Hence $\lambda(\omega[i]) q(\omega) = \lambda(\omega'[j])q(\omega')$. Finally, in the remaining case the $\1$ inserted in \dg~is not labelled since it is in a block. It is straightforward to see that the insertion of \dg~does not change other labels, so $\lambda(\omega[i]) q(\omega) = \lambda(\omega'[j])q(\omega')$.

\item[(L2)] In this case we remove the \st, which causes $n_{y^*}(\omega') = 0$, and have $\lambda(\omega[i]) = \alpha_*$. We have the same three cases as above, but instead of $n_y(\omega) = n_y(\omega') + 1$ and multiplying $q(\omega)$ by $\alpha$ we have $n_{y^*}(\omega) = n_{y^*}(\omega') + 1$ and multiply $q(\omega)$ by $\alpha_*$. Hence the same analysis works here.

\item[(L3)] Before the transition the 1 in the second column $\wb$ does not have a label as it is not left of the leftmost 0. The transition only introduces a $1/\alpha_*$, and $\lambda(\omega'[j]) = \alpha_*$ as $j = i$. Hence $\lambda(\omega[i]) q(\omega) = \lambda(\omega'[j])q(\omega')$, since $\lambda(\omega[i]) = 1$.

\end{enumerate}

\item Right border:

\begin{enumerate}

\item[(R1)] In this case we remove the \bw~to the left of \st. This removes a label $z$, so $n_z(\omega) = n_z(\omega') + 1$. Then, a \wb~is inserted to the right of $j = j_1$. Note that this has no effect on the labels of other particles. There are three cases: \bq, \xx, or \st~to the left of $j$. The second introduces no new label since there is a 0-column to the left. The claim follows from $\lambda(\omega'[j]) = 1$. The third case introduces a label $y$. Then $n_y(\omega') = n_y(\omega) + 1$, and $\lambda(\omega[i]) q(\omega) = \lambda(\omega'[j]) q(\omega')$, since $\lambda(\omega[i]) = \beta, \lambda(\omega'[j]) = \alpha$. In the first $? = 1$ and the 1 in the inserted \wb~is in a block, or $? = \1$ and there is a label $z$ to the left. Hence no new label is introduced in this case, and $\lambda(\omega[i]) q(\omega) = \lambda(\omega'[j]) q(\omega')$, since $\lambda(\omega'[j]) = 1$.

\item[(R2)] This case is similar to the one above (replace $\beta$ by $\beta_*$ and $z$ by $z^*$).

\item[(R3)] Before the transition the $\1$ in the penultimate column $\bw$ has label $z'$. The transition only introduces a $1/\beta_*$, and $\lambda(\omega'[j]) = \beta_*$ as $j = i$. Hence $\lambda(\omega[i]) q(\omega) = \lambda(\omega'[j])q(\omega')$.

\end{enumerate}

\end{itemize}

We have thus shown the claim to be true in each case separately, thereby completing the proof.
\end{proof}

We are now in a position to prove the formula for the stationary distribution.

\begin{proof}[\textbf{Proof} of Theorem \ref{thm:dist_d}.] 
By Proposition~\ref{prop:tworow-irred}, the two-row $D^*$ process has a unique stationary distribution. 
Therefore, it suffices to show that $\hat{\pi}^*$ satisfies the balance equation,
\begin{equation}
\label{balance}
\hat{\pi}^*(\omega) = \sum_{\omega' \in \widehat{\Omega}^*_{n, n_0}} \mathbb{P}( \omega' \to \omega)  \hat{\pi}^*(\omega').
\end{equation}
By definition of the process, $\mathbb{P}( \omega' \to \omega) \neq 0$ if and only if $T_*(\omega',i) = \omega$ for some $i$, and if so, $\mathbb{P}( \omega' \to \omega) = \lambda(\omega'[i])/(n-1)$. For each $i$, there is also the possibility that no transition occurs with probability $1 - \lambda(\omega[i])$ if we are already in state $\omega$.
Therefore, we can rewrite the right hand side of the balance equation \eqref{balance} as
\[
\frac{1}{n-1}\sum_{i=1}^{n-1} 
\sum_{\substack{\omega' \in \widehat{\Omega}^*_{n, n_0} \\ T_*(\omega',i) = \omega}}  \lambda(\omega'[i])  \hat{\pi}^*(\omega')
+ \frac{1}{n-1}\sum_{i=1}^{n-1}  (1 - \lambda(\omega[i])) \hat{\pi}^*(\omega).
\]
But by Proposition~\ref{prop:tworow-bij}, we know that the map $\bar{T}_*$ taking $(\omega',i) \mapsto (\omega,j)$ is a bijection and hence we can rewrite the first sum as
\[
\frac{1}{n-1} \sum_{\substack{ j=1 \\ (\omega',i) = \bar{T_*}^{-1}(\omega, j)}}^{n-1} \lambda(\omega'[i])  \hat{\pi}^*(\omega')
+ \frac{1}{n-1}\sum_{i=1}^{n-1}  (1 - \lambda(\omega[i])) \hat{\pi}^*(\omega).
\]
Now, by Lemma~\ref{lem:transfer_d}, the summand in the first sum is $ \lambda(\omega[j]) \hat{\pi}^*(\omega)$. Therefore, we obtain
\[
\frac{1}{n-1} \sum_{j=1}^{n-1} \lambda(\omega[j]) \hat{\pi}^*(\omega) 
+ \frac{1}{n-1}\sum_{j=1}^{n-1}  (1 - \lambda(\omega[j])) \hat{\pi}^*(\omega),
\]
which is easily seen to sum up to $\hat{\pi}^*(\omega)$, giving the balance equation as desired.
\end{proof}

For $\omega \in \widehat{\Omega}^*_{n, n_0}$, let $\omega_1$ denote the top row of $\omega$. From Propositions~\ref{prop:lump-ss} and~\ref{prop:tworow-lump}, we immediately obtain a combinatorial formula for the stationary probabilities in the \dstar{}.
Recall that the stationary probabilities in the latter are denoted by $\pi^*$.

\begin{cor}
\label{cor:tworow}
The stationary probability of a configuration $\tau \in \Omega^*_{n,n_0}$ in the \dstar{} is given by 
\[
\pi^*(\tau) = \sum_{\substack{\omega \in \widehat{\Omega}^*_{n, n_0} \\ \omega_1 = \tau}} \hat{\pi}^*(\omega).
\]
\end{cor}

\section{Correlations in MultiTASEPs}
\label{sec:part}
\subsection{Limiting direction for type $\chC$}
\label{sec:cpart}

To understand correlations in the \cmt{}, it will suffice to consider the \ct{}. By Theorem \ref{thm:lump} we can use the \dstar{} with $\alpha_*=\beta_*=0$. Removing the $*$'s at the first and last sites, this is equivalent to 
a model solved in~\cite{arita-2006}, and is known as the {\em semipermeable exclusion process}. We will in this subsection borrow results from there instead of using the two-row $D^*$-process. 
We can obtain our results by setting $\alpha = \beta = 1$. 

The semipermeable exclusion process is ergodic and thus has a unique stationary distribution. Various properties of the process are known due to work of Arita~\cite{arita-2006}. The physics of the model has been studied in~\cite{als-2009}.

First recall that the ballot numbers $\ballot nk$ are given by
\begin{equation}
\label{def-ballot}
\ballot nk = \binom{n+k}n - \binom{n+k}{n+1} = \frac{n-k+1}{n+1} \binom{n+k}n, \quad 0 \leq k \leq n. 
\end{equation}
The ballot numbers $\ballot{n}{k}$ count the number of up-right paths
from $(0,0)$ to $(n,n)$ which stay on or below the diagonal $x=y$ and which touch 
the diagonal $n-k+1$ times (counting both endpoints).
The first few rows of the triangular array of ballot numbers are as follows:
\[
\begin{array}{ccccc}
1 \\
1 & 1 \\
1 & 2 & 2 \\
1 & 3 & 5 & 5 \\
1 & 4 & 9 & 14 & 14 \\
\end{array}
\]
The array satisfies the Pascal triangle-like recurrence 
\begin{equation}
\label{ballot-numbers-recurrence}
\ballot{n}{k} = \ballot{n-1}{k} + \ballot{n}{k-1}, \quad 0 < k < n.
\end{equation}
The last two diagonals $C^n_n = C^n_{n-1}$ give the $n$'th {\em Catalan number}, $\catalan n = \frac{1}{2n+1} \binom{2n}{n}$.

\begin{theorem}[{\cite[Eq.~(27)]{arita-2006}}]
\label{thm:semiperm-pf}
For $\alpha \neq \beta$, the partition function of the semipermeable exclusion process is given by
\[
Z_{n,n_0}(\alpha,\beta) = \sum_{k=0}^{n-n_0} \ballot{n+n_0-1}{n-n_0-k}
\frac{\beta^{-k} - \alpha^{-k}}{\beta^{-1} - \alpha^{-1}}.
\]
For $\alpha = \beta$, the partition function is given by
\[
Z_{n,n_0}(\alpha,\alpha) = \sum_{k=0}^{n-n_0} (k+1) \ballot{n+n_0-1}{n-n_0-k} \alpha^{-k}.
\]
\end{theorem}

\begin{cor}
\label{cor:semiperm-pf}
When $\alpha = \beta = 1$, the partition function is
\[
Z_{n,n_0} (1,1) = \ballot{n+n_0+1}{n-n_0}.
\]
\end{cor}

\begin{proof}
Plug in $\alpha = 1$ in the second formula for the partition function in Theorem~\ref{thm:semiperm-pf},
\begin{align*}
Z_{n,n_0} (1,1) &= \sum_{k=0}^{n-n_0} (k+1) \ballot{n+n_0-1}{n-n_0-k} \\
&= \sum_{k=0}^{n-n_0} (k+1) \left( \binom{2n-k-1}{n+n_0-1} - 
\binom{2n-k-1}{n+n_0} \right).
\end{align*}
Now, use the fact that $k + 1 = \binom{k+1}{1}$ and the `dual' Chu-Vandermonde identity,
\[
\sum_{m=0}^p \binom{m}{j} \binom{p-m}{\ell-j} = \binom{p+1}{\ell+1},
\quad \text{valid for all $0 \leq j \leq \ell$},
\]
to evaluate both binomial sums. Then
\[
Z_{n,n_0} (1,1) = \binom{2n+1}{n+n_0+1} - \binom{2n+1}{n+n_0+2},
\]
which gives the desired result.
\end{proof}

We want to calculate the probability in the stationary distribution that the last site is occupied by a $1$. Fortunately for us, this has also been computed by Arita.

\begin{theorem}[{\cite[Eq.~(38)]{arita-2006}}]
\label{thm:semiperm-dens}
The probability of the $j$'th site being occupied by a $1$ in the semipermeable exclusion process with $n$ sites and $n_0$ $0$'s is given by
\[
\mathbb{P}^{n,n_0}_{\alpha,\beta}(w_j = 1) =  \sum_{i=0}^{n-j-1} \catalan i \frac{Z_{n-i-1,n_0}(\alpha,\beta)}{Z_{n,n_0}(\alpha,\beta)}
+ \frac{Z_{j-1,n_0}(\alpha,\beta)}{Z_{n,n_0}(\alpha,\beta)} \sum_{k=0}^{n-j} \frac{\ballot{n-j-1}{n-j-k}}{ \beta^{k+1}}.
\]
\end{theorem}

Setting $j = n$ and $\alpha = \beta = 1$ in Theorem~\ref{thm:semiperm-dens} leads using Corollary~\ref{cor:semiperm-pf} and the  fact that $\ballot{-1}{0} = 1$ to
\begin{equation}
\label{c-prob-last-site}
\mathbb{P}^{n,n_0}_{1,1}(w_n = 1) = \frac{\ballot{n+n_0}{n-1-n_0}}{\ballot{n+n_0+1}{n-n_0}} = \frac{(n-n_0)(n+n_0+2)}{2n(2n+1)}.
\end{equation}
We are now in a position to prove our main result for type $C$.

\begin{theorem}
\label{thm:typec-dens}
The probability that the last site is occupied by $i$ in the \cmt{} of size $n$ is given by
\[
\aver{i}_{\chC} = \frac{2i+1}{2n(2n+1)}.
\]
\end{theorem}

\begin{proof}
We use the coloring argument from Section~\ref{sec:color}. 
Using the $k$-coloring procedure, we can lump the \cmt{} to the \ct{} with $n$ sites and $k-1$ $0$'s
by Proposition~\ref{prop:kcoloring} for $1 \leq k \leq n$. If the last site is an $i$ in the \cmt{}, 
it will become a $0$ (resp. $1$) if $k > i$ (resp. $k \le i$) in the \ct{}. 
It then follows that the stationary probability in \cmt{} that
the last site is occupied by a particle of species $i$ is the difference of the stationary probabilities in the \ct{} with 
$i-1$ $0$'s and $i$ $0$'s. As a result, we obtain
\begin{align*}
\aver{i}_{\chC} &= \mathbb{P}^{n,i-1}_{1,1}(w_n = 1)
-\mathbb{P}^{n,i}_{1,1}(w_n = 1) \\
&= \frac{(n-i+1)(n+i+1)}{2n(2n+1)} - \frac{(n-i)(n+i+2)}{2n(2n+1)},
\end{align*}
using \eqref{c-prob-last-site}, which proves the result.
\end{proof}

The result for the limiting direction of Lam's random walk in Theorem~\ref{thm:lim-dir-c} now follows from Theorem~\ref{thm:typec-dens}.

\subsection{Partition function and correlations in \bt{}}
\label{sec:bpart}
In this section we use the combinatorial description of the stationary distribution of the two-row $D^*$ process explained in Section~\ref{sec:duchi-schaeffer} to compute the partition function, and the probabilities of the entries of the last two positions in a state.

Recall the ballot numbers  $C_n^m$ from \eqref{def-ballot}.
They also enumerate Dyck paths from the origin to $(2m+2,0)$ with $m-n$ intermediate returns to the $x$-axis. 
Recall also that $C_n = C^n_n = C^n_{n-1}$ is the $n$'th Catalan number.

A Motzkin path of length $k$ is a path from $(0,0)$ to $(k,0)$ that does not go below the $x$-axis and uses steps $(1,1),(1,-1)$ and $(1,0)$. There is a translation from the two-row model to bicolored Motzkin paths 
(in which the horizontal steps can have two different colors). For non-zero and non-star columns we translate to steps in a Motzkin path as follows:
\[
\begin{tikzpicture}[scale=0.5]
\begin{scope}
\draw(0,1)node{$1$}(0,0)node{$1$};
\draw[->](1,0)--(2,1);
\draw(5,1)node{$\1$}(5,0)node{$\1$};
\draw[->](6,1)--(7,0);
\draw(10,1)node{$\1$}(10,0)node{$1$};
\draw[->, dashed, red](11,0.5)--(12,0.5);
\draw(15,1)node{$1$}(15,0)node{$\1$};
\draw[->, thick, blue](16,0.5)--(17,0.5);
\end{scope}
\end{tikzpicture}
\]
This map is also used in \cite[Section 7]{duchi-schaeffer-2005}.
A configuration in the two-row $D^*$ process without zeros corresponds directly to a bicolored Motzkin path (ignoring the star columns in the ends). In terms of Motzkin paths, the weights can be described as follows.
The weight is $1/\beta$ on all horizontal steps of the second color that are on the $x$-axis. 
The weight $1/\al$ is assigned to all up-steps from the $x$-axis and horizontal steps of the first color on the $x$-axis, with the extra condition that no $1/\al$-weights 
are counted if they are to the right of a $1/\beta$-weight. The weight of a path is the product of its step weights.

There is a well-known bijection from bicolored Motzkin paths of length $k$ to Dyck paths of length $2k+2$, where the Dyck path starts with an up and ends with a down-step, the up and down-steps of the Motzkin path are doubled, and the horizontal 
steps are mapped to up-down and down-up, respectively. Here, the color which has label $1/\alpha$ corresponds to an up-down step. 
Using this bijection, Dyck paths will have the following weights: every down-step to the $x$-axis will be weighted $1/\beta$, except for the last one, and every up-step from the horizontal line $x=1$ with no preceding $1/\beta$ will be weighted $1/\alpha$. 
The number of bicolored Motzkin paths of length $k$ with $i$ steps of weight $1/\beta$ is thus $C^k_{k-i}$, using the interpretation above.

Let $V_k(\al, \beta)$ be the generating function for bicolored Motzkin paths of length $k$ with weights as above. For example, 
\[
V_2(\al,\beta)=\al^{-2}+\beta^{-2}+\al^{-1}\beta^{-1}+\al^{-1}+\beta^{-1};
\] 
see Figure \ref{F:bijection}.
Let $M_k(\beta) :=V_k(1,\beta)$. 
Then the bijection to Dyck paths of length $2k+2$ gives $V_k(1,1)=\catalan{k+1}$ and 
\[
M_k(\beta)=\sum_{i=0}^k C_{k-i}^k\beta^{-i}.
\]

\begin{figure}
\centering
\begin{tikzpicture}[scale=0.5]
\put(0,0)
{
\draw[->](0,0)--(1,1);
\draw[->](1,1)--(2,0);
\draw[red](0,1)node{$\alpha$};
\draw[->, dashed, red](5,0)--(6,0);
\draw[->, dashed, red](6,0)--(7,0);
\draw[red](5.5,0.5)node{$\alpha$}(6.5,0.5)node{$\alpha$};
\draw[->, dashed, red](10,0)--(11,0);
\draw[->, thick, blue](11,0)--(12,0);
\draw[red](10.5,0.5)node{$\alpha$}[blue](11.5,0.5)node{$\beta$};
\draw[->, thick, blue](14,0)--(15,0);
\draw[->, dashed, red](15,0)--(16,0);
\draw[blue](14.5,0.5)node{$\beta$};
\draw[->, thick, blue](18,0)--(19,0);
\draw[->, thick, blue](19,0)--(20,0);
\draw[blue](18.5,0.5)node{$\beta$}[blue](19.5,0.5)node{$\beta$};
}
\put(0,-30){
\draw[->](0,0)--(0.5,0.5);
\draw[->](0.5,0.5)--(1,1);
\draw[->](1,1)--(1.5,1.5);
\draw[->](1.5,1.5)--(2,1);
\draw[->](2,1)--(2.5,0.5);
\draw[->](2.5,0.5)--(3,0);
\draw[->](5,0)--(5.5,0.5);
\draw[->](5.5,0.5)--(6,1);
\draw[->](6,1)--(6.5,0.5);
\draw[->](6.5,0.5)--(7,1);
\draw[->](7,1)--(7.5,0.5);
\draw[->](7.5,0.5)--(8,0);
\draw[->](10,0)--(10.5,0.5);
\draw[->](10.5,0.5)--(11,1);
\draw[->](11,1)--(11.5,0.5);
\draw[->](11.5,0.5)--(12,0);
\draw[->](12,0)--(12.5,0.5);
\draw[->](12.5,0.5)--(13,0);
\draw[->](14,0)--(14.5,0.5);
\draw[->](14.5,0.5)--(15,0);
\draw[->](15,0)--(15.5,0.5);
\draw[->](15.5,0.5)--(16,1);
\draw[->](16,1)--(16.5,0.5);
\draw[->](16.5,0.5)--(17,0);
\draw[->](18,0)--(18.5,0.5);
\draw[->](18.5,0.5)--(19,0);
\draw[->](19,0)--(19.5,0.5);
\draw[->](19.5,0.5)--(20,0);
\draw[->](20,0)--(20.5,0.5);
\draw[->](20.5,0.5)--(21,0);
};
\end{tikzpicture}
\vskip1cm
\caption{The five bicolored Motzkin paths of length $k=2$ and the corresponding Dyck paths of length 6.}
\label{F:bijection}
\end{figure}
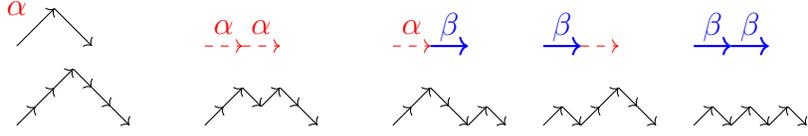

We will need a sequence of identities which we formulate as lemmas. 
They are probably not new, but we have not found good references.

\begin{lemma} 
\label{L:Cat1} 
For $a\ge b\ge j\ge 0$, we have 
\[
\sum_{i=j}^b C_{b-i}^{a-i}\cdot C_{i-j}^i=C_{b-j}^{a+1}.
\]
\end{lemma}

\begin{proof}
The right-hand side is also the number of non-negative NE-SE paths from $(0,0)$ to $(a+1+b-j,a+1-b+j)$. The left-hand side counts the same, but summing over the last time the path has $y$-coordinate $a-b$, which happens in position $(a+b-2i,a-b)$ for $j \leq i \leq b$. The number of paths to this position is $C_{b-i}^{a-i}$. We then take one NE-step and the number of possible ways to continue from there is $C_{i-j}^i$.
\end{proof}

\begin{lemma} 
\label{L:Cat2} 
For $n\ge b$ and $b-d-a \ge 0$, 
we have 
\[
\sum_{i=0}^{n-b} C_{i}^{i+d} \binom{2n-2i-d-a}{n-b-i}=\binom{2n-a+1}{n-b}.
\]
\end{lemma}

\begin{proof}
The right-hand side counts NE-SE lattice paths from $(0,0)$ to $(2n-a+1,2b-a+1)$. On the left-hand side we are summing over the last position $(2n-2i-d-a, 2b-d-a)$ where the path is below the line $y = 2b-d-a+1$. Note that $\binom{2n-2i-d-a}{n-b-i}$ counts all NE-SE paths from $(0, 0)$ to $(2n-2i-d-a, 2b-d-a)$, and $C_{i}^{i+d}$, after shifting, NE-SE paths from $(2n-2i-d-a+1, 2b-d-a+1)$ to $(2n-a+1,2b-a+1)$ staying weakly above the line $y = 2b-d-a+1$. There is a northeast-step between the two parts.
\end{proof}

The next result appears as \cite[Equation (A11)]{dehp}, but we restate it because we will use the proof strategy later.

\begin{lemma}
\label{L:Vk} 
For $k\ge 0$, we have 
\[
V_k(\al,\beta)=\sum_{i=0}^k\sum_{j=0}^{k-i} C_{k-i-j}^{k-1}\al^{-i}\beta^{-j}.
\]
\end{lemma}

\begin{proof} 
We use the interpretation of Dyck paths of length $2k+2$. A standard recursion for $V_k$ is given by letting $i$ be the first index for which step $2i+2$ is on the $x$-axis and thus labelled $1/\beta$ if $i<k$. Further, let $j$ be the number of intermediate returns to the line $x=1$ by the path before $2i-1$, and $r$ the number of intermediate returns to the $x$-axis by the path after $2i+2$. We get the equation
\begin{align*}
V_{k}(\alpha, \beta) =& \sum_{j=0}^{k-1}\alpha^{-j-1}C_{k-1-j}^{k-1} \\
&+ \sum_{i=0}^{k-1}\beta^{-1}  \sum_{r=0}^{k-i-1}  \beta^{-r}C_{k-i-1-r}^{k-i-1}\sum_{j=0}^{i-1}\alpha^{-j-1}C_{i-1-j}^{i-1} \\
=&  \sum_{r=0}^{k-1}\sum_{j=0}^{k-1}  \beta^{-r-1}\alpha^{-j-1}  \sum_{i=j+1}^{k-r-1}  C_{k-i-1-r}^{k-i-1}C_{i-1-j}^{i-1}\\
&+   \sum_{j=0}^{k-1}(\beta^{-j-1}  +  \alpha^{-j-1})C_{k-1-j}^{k-1}.
\end{align*} 
The statement now follows from Lemma \ref{L:Cat1}. 
\end{proof}

\begin{lemma} 
\label{L:Mk} 
For $k\ge 0$, we have $M_k(\frac{1}{2})=\binom{2k+1}{k}$ and $V_k(\frac{1}{2},\frac{1}{2})=2^{2k}$.
\end{lemma}

\begin{proof} 
For the first identity, we note that it is clearly true for $k=0$. 
We have 
\[
M_k\left( \frac{1}{2} \right) =\sum_{i=0}^k C_{k-i}^k 2^{i}.
\] 
If we let $2i$ be the first return of the path to the $x$-axis, we get the recursion 
\[
M_k\left( \frac{1}{2} \right) =\catalan k+\sum_{i=1}^k \catalan {i-1}\cdot 2\cdot M_{k-i} \left( \frac{1}{2} \right) .
\] 
By induction this gives 
\[
M_k\left( \frac{1}{2} \right) =\catalan k+2\sum_{i=1}^k \catalan {i-1} \binom{2(k-i)+1}{k-i},
\] 
which by Lemma \ref{L:Cat2} becomes 
\[
M_k\left( \frac{1}{2} \right) =\catalan k+2\binom{2k}{k-1}=\binom{2k+1}{k}.
\]
For the second identity, we use that in the proof of Lemma \ref{L:Vk} we get 
\begin{align*}
V_k(\alpha,\beta) &= \sum_{i=0}^{k-1}  \sum_{r=0}^{k-1-i}  \beta^{-r-1}C_{k-i-1-r}^{k-i-1}\sum_{j=0}^{i-1}\alpha^{-j-1}C_{i-j-1}^{i-1} +   \sum_{j=0}^{k-1}\alpha^{-j-1}C_{k-1-j}^{k-1}\\ 
&= \sum_{i=0}^{k-1}  \beta^{-1}M_{k-i-1}(\beta)\alpha^{-1}M_{i-1}(\alpha)+\alpha^{-1}M_{k-1}(\alpha),
\end{align*}
which by the first identity gives
\begin{align*}
V_k\left( \frac{1}{2},\frac{1}{2} \right) &= \sum_{i=0}^{k-1}2\binom{2k-2i-1}{k-1-i}2\binom{2i-1}{i-1}+2\binom{2k-1}{k-1} \\
&= \sum_{i=0}^{k}\binom{2k-2i}{k-i}\binom{2i}{i}=2^{2k},
\end{align*}
which completes the proof.
\end{proof}

Now we start analysing the two-row model in more detail. First, note that for $k$ consecutive columns with no zeros and $\al=\beta=1$ there are $V_k(1,1)=\catalan {k+1}$ possible two-row configurations~\cite{duchi-schaeffer-2005}. We use this as base case for the following inductive proof.

\begin{lemma} 
\label{L:endinzero} 
For $k$ consecutive columns in the two-row model with $n_0$ zeros and $\al=\beta=1$ the (weighted, each with weight 1) number of possible configurations is $C_{k-n_0}^{k+n_0+1}.$
\end{lemma}

\begin{proof} 
We perform induction on $n_0$.
Let $i$ be the position of the last zero column. Then by induction the number of configurations is 
\[
\sum_{i=n_0}^k C_{i-1-(n_0-1)}^{i-1+(n_0-1)+1} \catalan {k-i+1}=\sum_{i=2n_0-1}^{k+n_0-1} C_{i-2n_0+1}^{i} C_{k+n_0-i-1}^{k+n_0-i},
\] 
since $\catalan {n}=C_{n-1}^{n}$. The statement now follows from Lemma \ref{L:Cat1}.
\end{proof}

Now we are in a position to compute the partition function for the \bt{}, which we will denote in this section by $Z_{n,n_0}$.

\begin{theorem} 
\label{T:partitionB} 
For any $n\ge n_0\ge 0$, the partition function for the \bt{} is $Z_{n,n_0}=\binom{2n}{n-n_0}$.
\end{theorem}

\begin{proof} 
Recall that in the \bt{} we have $\alpha_* = 0$, $\alpha = 1$ and $\beta = \beta_* = \frac{1}{2}$.
If $n_0=0$, then a two-row configuration must end with stars and thus $Z_{n,0}=2M_{n-1}(\frac{1}{2})=2\binom{2n-1}{n-1}=\binom{2n}{n}$, by Lemma \ref{L:Mk}. The number 2 is the weight of the star at the end.
For $n_0>0$, let $i$ be the position of the last 0-column. Since the $\beta$-weights only contribute to the right of the rightmost zero column, the total weight is, using Lemma \ref{L:endinzero}, given by
\begin{align*}
Z_{n,n_0} &= C_{n-1-(n_0-1)}^{n-1+(n_0-1)+1} + \sum_{i=n_0}^{n-1} C_{i-n_0}^{i+n_0-1} M_{n-i-1} \left( \frac{1}{2} \right) \cdot 2 \\
&= C_{n-n_0}^{n+n_0-1} + 2\sum_{i=n_0}^{n-1} C_{i-n_0}^{i+n_0-1} \binom{2n-2i-1}{n-i-1},
\end{align*} 
where the first term corresponds to $i=n$. 
Using Lemma \ref{L:Cat2}, we obtain
\[
Z_{n,n_0}=C_{n-n_0}^{n+n_0-1} + 2 \binom{2n-1}{n-n_0-1}=\binom{2n}{n-n_0},
\]
proving the result.
\end{proof}

We fix $n$ and $n_0$ such that $n\ge n_0\ge 0$.
In this subsection, let ${\aver {i} }_{n,n_0}$ and ${\aver {i,j}}_{n,n_0}$ denote the stationary probability of having an $i$ in the last position and $i,j$ in the last two positions, respectively, in the \bt{} on $n$ sites with $n_0$ $0$'s.

\begin{theorem} 
\label{T:Bprob} 
For any $n\ge n_0\ge 0$ and $\beta=1/2$, we have the following table for 
$Z_{n,n_0}\cdot\aver {i,j}_{n,n_0}$ in the \bt{}.

\begin{center}

\renewcommand{\arraystretch}{2.5} 
\begin{tabular}{| c | c | c | c| }
\hline
$i\backslash j$ & $\1$ & $0$ & $1$ \\ 
\hline
$\1$ & $\displaystyle \binom{2n-2}{n-n_0-2}$ & $\displaystyle C_{n-n_0-1}^{n+n_0-1}$ & $\displaystyle \binom{2n-2}{n-n_0-2}$ \\
\hline
$0$ & $\displaystyle C_{n-n_0-1}^{n+n_0-2}$& $\displaystyle C_{n-n_0}^{n+n_0-3}$ & $\displaystyle C_{n-n_0-1}^{n+n_0-2}$\\
\hline
$1$ & $\displaystyle 2\binom{2n-3}{n-n_0-2}$ & $\displaystyle C_{n-n_0-1}^{n+n_0-2}$ & $\displaystyle 2\binom{2n-3}{n-n_0-2}$ \\
\hline
\end{tabular}

\end{center}
\end{theorem}

\begin{proof} 
From the lumping to the \dstar{} by Theorem~\ref{thm:lump}(2), we know that $1$ and $\1$ have the same probability of being at the last site, and so the corresponding columns are equal. 
By the computation in the proof of Theorem \ref{T:partitionB}, the column sums are $\binom{2n-1}{n-n_0-1}$ and for column $0$ the sum is $C_{n-n_0}^{n+n_0-1}$. 
By Lemma \ref{L:endinzero} we get directly that $Z_{n,n_0}\aver {0,0}_{n,n_0} =C_{n-2-(n_0-2)}^{n-2+(n_0-2)+1}$ and 
$Z_{n,n_0}\aver{0,\1}_{n,n_0} = C_{n-2-(n_0-1)}^{n-2+(n_0-1)+1}$ as desired. 
A word ending in $10$ must have a bottom row ending in $\1 0$ in the two-row $D^*$ configuration, which means that again Lemma \ref{L:endinzero} is applicable and we get $Z_{n,n_0}\aver {1,0}_{n,n_0}=C_{n-2-(n_0-1)}^{n-2+(n_0-1)+1}$. We can then also deduce that 
\[
Z_{n,n_0}\aver {\1,0}_{n,n_0}=Z_{n,n_0}({\aver 0}_{n,n_0} -\aver{1,0}_{n,n_0}-\aver{0,0}_{n,n_0}) =C_{n-n_0-1}^{n+n_0-1}
\]
using the recursion for ballot numbers.
A word ending in $1*$ in the \dstar{} must have a bottom row ending in $\1 *$ in the two-row $D^*$ configuration. We get a factor $1/\beta=2$ from the $\1$ in the bottom row, and the factor $1/\beta_*=2$ cancels by
the fact that both $1$ and $\1$ lump to $*$. Letting $i$ be the position of the last 0, we get, as in the proof of Theorem \ref{T:partitionB},
\begin{align*}
Z_{n,n_0}\aver {1,\1}_{n,n_0} &= \sum_{i=n_0}^{n-2} C_{i-n_0}^{i+n_0-1} M_{n-i-2} \left( \frac{1}{2} \right) \cdot 2 \\
&= 2\sum_{i=n_0}^{n-2} C_{i-n_0}^{i+n_0-1} \binom{2n-2i-3}{n-i-2}=2\binom{2n-3}{n-n_0-2},
\end{align*}
where the last equality follows from Lemma \ref{L:Cat2}. 
Finally,
\begin{align*}
Z_{n,n_0}\aver{\1,\1}_{n,n_0} &= Z_{n,n_0}({\aver \1}_{n,n_0} -\aver{0,\1}_{n,n_0} -\aver{1,\1}_{n,n_0}) \\
&= \binom{2n-1}{n-n_0-1}- \left( C_{n-n_0-1}^{n+n_0-2}+2\binom{2n-3}{n-n_0-2} \right) \\
&= \binom{2n-2}{n-n_0-2},
\end{align*}
completing the proof.
\end{proof}

\subsection{Limiting direction for type $\hatB$}
\label{sec:limB}

In this section, we will prove formulas for the two-point correlations for the \bmt{}
using the two-point correlations in Section \ref{sec:bpart} for the \bt{}.
As before $\aver {i,j}_{B}$ means the stationary probability of having a particle of species $i$ at site $n-1$
and a particle of species $j$ at site $n$ in the \bmt{} on $n$ sites. 
To avoid confusion, we will reuse the notation $\aver{i,j}_{n,n_0}$ as in Theorem~\ref{T:Bprob} to mean the corresponding probability in the \bt{} with $n$ sites and $n_0$ $0$'s. We will continue to use $Z_{n,n_0}$ for the partition function of the latter.

We will use the lumping described in Proposition \ref{prop:kcoloring}. There are only $n$ different such lumpings and this is 
not enough to determine all the $4n^2$ two-point correlations $\aver {i,j}_{B}$; see Table \ref{tab:bvalues-n=4}. 
But it turns out that we 
can get just enough information to prove Theorem \ref{thm:lim-dir-b}. As usual, we identify $-k$ and $\neg k$. There are no zeros in the \bmt{} but for ease of notation we write $\aver {j,0}_{B}=\aver {0,j}_{B}=0$. We define row and column sums, and up and down-hooks as follows. 
\begin{align*}
\Col_i(n) &:=\sum_{j=-n}^n \aver {j,i}_{B},\quad -n\le i\le n\\
\Row_i(n) &:=\sum_{j=-n}^n \aver {i,j}_{B},\quad -n\le i\le n \\
\Hd_i(n) &:=\sum_{j=i+1}^n \aver {i,-j}_{B}+ \aver {j,-i}_{B},\quad 1\le i\le n \\
\Hu_i(n) &:=\sum_{j=i+1}^n \aver {-j,i}_{B}+ \aver {-i,j}_{B},\quad 1\le i\le n.
\end{align*}
See Figure \ref{F:hooks} for an illustration.

\begin{figure}[htbp!]
\begin{subfigure}{0.45\textwidth}
\centering
\arrayrulewidth=1px
\resizebox{1.0\textwidth}{!}{$\begin{array}{c|c|c|c|c|c|c|c|c|c|}
\multicolumn{1}{c}{} & \multicolumn{1}{c}{-n} & \multicolumn{1}{c}{\dots} & \multicolumn{1}{c}{-(i+1)} & \multicolumn{1}{c}{-i} & \multicolumn{1}{c}{\dots} & \multicolumn{1}{c}{-1} & \multicolumn{1}{c}{1} & \multicolumn{1}{c}{\dots} & \multicolumn{1}{c}{n}\\[0.1cm]
\hhline{~|*{9}{-}|}
&&&&&&&&& \\[-0.3cm]
-n &&&&& & & & &  \\[0.1cm] \hhline{~|*{9}{-}|}
\vdots &&&&& & & & &  \\[0.1cm] \hhline{~|*{9}{-}|}
&&&&&&&&& \\[-0.3cm]
-1 & &  & & & & &&&\\[0.1cm] \hhline{~|*{9}{-}|}
&&&&&&&&& \\[-0.3cm]
1 & &  & & & & &&&\\[0.1cm] \hhline{~|*{9}{-}|}
\vdots &&&&& & & & & \\[0.1cm] \hhline{~|*{9}{-}|}
& \cellcolor{gray!50} & \cellcolor{gray!50} & \cellcolor{gray!50} &&&&&& \\[-0.3cm]
i & \cellcolor{gray!50} & \cellcolor{gray!50} & \cellcolor{gray!50} & & & &&&\\[0.1cm] \hhline{~|*{9}{-}|}
&&&& \cellcolor{gray!50}&&&&& \\[-0.3cm]
i+1 & & & & \cellcolor{gray!50} & & &&&\\[0.1cm] \hhline{~|*{9}{-}|}
\vdots & & & & \cellcolor{gray!50} & &&&& \\[0.1cm] \hhline{~|*{9}{-}|}
&&&&\cellcolor{gray!50}&&&&& \\[-0.3cm]
n & & & &  \cellcolor{gray!50} & &&&& \\[0.1cm] \hhline{~|*{9}{-}|}
\end{array}$}
\caption{$\Hd_i(n)$}
\end{subfigure}
\qquad
\begin{subfigure}{0.45\textwidth}
\centering
\arrayrulewidth=1px
\resizebox{1.0\textwidth}{!}{$\begin{array}{c|c|c|c|c|c|c|c|c|c|}
\multicolumn{1}{c}{} & \multicolumn{1}{c}{-n} & \multicolumn{1}{c}{\dots} & \multicolumn{1}{c}{-1} & \multicolumn{1}{c}{1} & \multicolumn{1}{c}{\dots} & \multicolumn{1}{c}{i} & \multicolumn{1}{c}{i+1} & \multicolumn{1}{c}{\dots} & \multicolumn{1}{c}{n}\\[0.1cm]
\hhline{~|*{9}{-}|}
&&&&&&\cellcolor{gray!50}&&& \\[-0.3cm]
-n &&&& & & \cellcolor{gray!50} & & & \\[0.1cm] \hhline{~|*{9}{-}|}
\vdots &&&& & & \cellcolor{gray!50} & & &\\[0.1cm] \hhline{~|*{9}{-}|}
&&&&&&\cellcolor{gray!50}&&& \\[-0.3cm]
-(i+1) &&&& & & \cellcolor{gray!50} & & & \\[0.1cm] \hhline{~|*{9}{-}|}
&&&&&&& \cellcolor{gray!50} & \cellcolor{gray!50}& \cellcolor{gray!50} \\[-0.3cm]
-i &&&& & & & \cellcolor{gray!50} & \cellcolor{gray!50} & \cellcolor{gray!50} \\[0.1cm] \hhline{~|*{9}{-}|}
\vdots & & & & & & &&&\\[0.1cm] \hhline{~|*{9}{-}|}
&&&&&&&&& \\[-0.3cm]
-1 & & &  & & & &&&\\[0.1cm] \hhline{~|*{9}{-}|}
&&&&&&&&& \\[-0.3cm]
1 & & &  & & & &&&\\[0.1cm] \hhline{~|*{9}{-}|}
\vdots & & & & & & &&&\\[0.1cm] \hhline{~|*{9}{-}|}
&&&&&&&&& \\[-0.3cm]
n & & & & & &&&& \\[0.1cm] \hhline{~|*{9}{-}|}
\end{array}$}
\caption{$\Hu_i(n)$}
\end{subfigure}
\caption{The (A) down-hook $\Hd_i(n)$ and (B) up-hook $\Hu_i(n)$ are shaded.}
\label{F:hooks}
\end{figure}

\begin{lemma}
\label{L:Corr-sums-B}
For any $n\ge 1$, we have in the \bmt{} on $n$ sites,
\begin{align*}
\Col_i(n) &= \frac{1}{2n}\quad -n\le i\le n,\ i \neq 0,\\
\Row_i(n) &= \begin{cases}
\ds \frac{1}{2n} & -n\le i\le -2,\\[0.2cm]
\ds \frac{n-1}{2n(2n-1)}& i=-1,\\[0.2cm]
\ds \frac{n^2+2n(2i-1)-3i^2-i+1}{2n(2n-1)(n-1)} & 1\le i\le n,\\
\end{cases}\\
\Hd_i(n) &= \frac{(n-i)(n+3i-1)}{2n(2n-1)(n-1)} \quad 1\le i\le n,\\
\Hu_i(n) &= \frac{n-i}{n(2n-1)} \quad 1\le i\le n.
\end{align*}
\end{lemma}

\begin{proof} 
Since ${\aver 1}_{n,k}=\sum_{i > k}{\aver i}_{B}$, we have 
${\aver i}_{B}={\aver 1}_{n,i-1}-{\aver 1}_{n,i}$. By Theorem \ref{T:Bprob} we can compute 
\[
{\aver 1}_{n,n_0}=\frac{\binom{2n-2}{n-n_0-2}+C_{n-n_0-1}^{n+n_0-2}+2\binom{2n-3}{n-n_0-2}}{Z_{n.n_0}}
=\frac{n-n_0}{2n}
\] 
using \eqref {def-ballot} and Theorem \ref{T:partitionB}. This gives  $\Col_i(n)={\aver i}_{B}=\frac{1}{2n}$.
Similarly, we can compute $\Row_i(n)$ by first computing the row sum for the \bt{} from Theorem \ref{T:Bprob}. 
We obtain
\[
{\aver {1,\cdot}}_{n,n_0} = \frac{2\binom{2n-3}{n-n_0-2}+C_{n-n_0-1}^{n+n_0-2}+2\binom{2n-3}{n-n_0-2}}{Z_{n.n_0}}
= \frac{(n^2-n_0^2)(2n-n_0-2)}{2n(2n-1)(n-1)}.
\]
Similar computations give  
${\aver {\1,\cdot}}_{n,n_0}=\frac{n-n_0}{2n}$ for $n_0\ge 1$ and ${\aver {\1,\cdot}}_{n,0}=\frac{n-1}{2n-1}.$
For $i\ge 2$ we get, from the same argument as above, 
\[
\Row_{-i}(n)={\aver {-i,\cdot}}_{B}={\aver {\1,\cdot}}_{n,i-1}-{\aver {\1,\cdot}}_{n,i}=\frac{1}{2n},
\] 
whereas 
\[
\Row_{-1}(n)=\frac{n-1}{2n-1}-\frac{n-1}{2n}=\frac{n-1}{2n(2n-1)}.
\] 
For any $i\ge 1$ we get 
\begin{multline*}
\Row_i(n)={\aver {1,\cdot}}_{n,i-1}-{\aver {1,\cdot}}_{n,i} \\
= \frac{(n^2-(i-1)^2)(2n-(i-1)-2)-(n^2-i^2)(2n-i-2)}{2n(2n-1)(n-1)},
\end{multline*}
which simplifies to the desired formula.
The $(i+1)$-coloring procedure (see Section~\ref{sec:color}) gives that ${\aver {1,\1}}_{n,i}=\sum_{j,k\ge i+1}{\aver{j,-k}}_{B}$, which means $\Hd_i(n) \allowbreak ={\aver {1,\1}}_{n,i-1}-{\aver {1,\1}}_{n,i}$.
By Theorem \ref{T:Bprob} this becomes 
\[
\Hd_i(n)=\frac{2\binom{2n-3}{n-(i-1)-2}}{Z_{n,i-1}}-\frac{2\binom{2n-3}{n-i-2}}{Z_{n,i}}
=\frac{(n-i)(n+3i-1)}{2n(2n-1)(n-1)}
\] 
as claimed.
Similarly 
\[
\Hu_i(n)={\aver {\1,1}}_{n,i-1}-{\aver {\1,1}}_{n,i}
=\frac{\binom{2n-2}{n-(i-1)-2}}{Z_{n,i-1}}-\frac{\binom{2n-2}{n-i-2}}{Z_{n.i}}=\frac{n-i}{n(2n-1)},
\]
completing the proof.
\end{proof}

Lemma~\ref{L:Corr-sums-B} is the key ingredient in the proof of the limiting direction of Lam's walk of type \hatB{} in Theorem~\ref{thm:lim-dir-b}.
From Lemma~\ref{L:Corr-sums-B} and the fact that ${\aver i}_{B}=\Col_i(n)$, we immediately arrive at the following 
curious observation, which might be of independent interest.

\begin{cor}
The probability of an $i$ at the last site in the \bmt{} on $n$ sites  is $1/(2n)$ for all $-n \leq i \leq n, i \neq 0$.
\end{cor}

Let $\pi_B(w_1=i)$ be the probability of having a particle of species $i$ in the first site of the \bmt{} on $n$ sites.
Another interesting curiosity is that the coefficient of $e_{k}$ in the type $B$ limiting direction is twice $\mathbb{P}(w_1=-k)$ for $1 \le k \le n$, as we shall see next.

\begin{theorem}
Let $n \geq 2$. Then
\[
\pi_B(w_1=k)= 
\begin{cases}
\ds \frac{2|k|-1}{2n(2n-1)} &  -n \le k \le -1,\\
\ds \frac{n^2 + n - 1}{2n(2n - 1)} & k = 1,\\
\ds \frac{1}{2n} & 1 < k \le n.
\end{cases}
\]
\end{theorem}

\begin{proof}
We use the coloring argument, which gives $\pi_B(w_1=k) = {\aver 1}_{n,k-1} - {\aver 1}_{n,k}$ for $k \geq 1$. 
Similarly, $\pi_B(w_1=-k) = {\aver{\neg{1}}}_{n,k-1} - {\aver{\neg{1}}}_{n,k}$ for $k \geq 1$.

Consider the two-row $D^*$ process. If the second position of the top row contains a $\neg{1}$, then the bottom row must contain a 1 in the same position. Hence 
\[
{\aver{\neg{1}}}_{n,k} = \frac{\binom{2n-2}{n-k-1}}{\binom{2n}{n-k}}
\] 
by Theorem \ref{T:Bprob} since the 1 does not contribute any extra weight. 
This yields, after some computations, the claim for $\pi_B(w_1=-k)$, $1 \le k \le n$.
Now, for ${\aver 1}_{n,k}$ we use ${\aver 1}_{n,k} = 1 - {\aver 0}_{n,k} - {\aver{\neg{1}}}_{n,k}$. For $k > 0$, 
\[
{\aver 0}_{n,k} = \frac{\binom{2n-2}{n-k}}{\binom{2n}{n-k}},
\] 
and for $k = 0$, ${\aver 0}_{n,0}  = 0$. Using these formulas for different values of $k$, one can compute the remaining cases.
\end{proof}

\subsection{Partition function and correlations in \dt{}}
\label{sec:dpart}

In this section, we will denote the partition function for the \dt{} by $Z_{n,n_0}$.

\begin{theorem} 
\label{T:partitionD} 
For $n\ge n_0\ge 1$, the partition function of the \dt{} is 
\[
Z_{n,n_0}=\sum_{j=0}^{n-n_0}\binom{2j}{j}\binom{2n-2j-2}{n-j-n_0}.
\] 
For $n_0=0,1$, we get $Z_{n,n_0} =4^{n-n_0}$.
\end{theorem}

\begin{proof} 
Recall that in the \dt{} we have $\alpha_* = \beta_* = \alpha = \beta = \frac{1}{2}$.
First, assume $n_0\ge 2$. We will appeal to the \dstar{} and use the lumping from Theorem~\ref{thm:lump}(3).
We divide configurations of the \dstar{} into four classes depending on the first and last sites.

\noindent
{\bf Case 1}: First we study all words starting and ending with 0. By Lemma \ref{L:endinzero},
 these configurations sum to $C_{n-n_0}^{n+n_0-3}$.

\noindent
{\bf Case 2}: The second case is all words starting with $0$ and ending in $*$. This is the same as the second case in the computation of the partition function $Z^B_{n-1,n_0-1}$ of the \bt{} in Theorem \ref{T:partitionB}, which is given by $2\binom{2n-3}{n-n_0-1}$.

\noindent
{\bf Case 3}: The third case is words starting with $*$ and ending in $0$, which by Case 2 and the symmetry in Proposition \ref{prop:ctdt-symm} also gives $2\binom{2n-3}{n-n_0-1}$. 

\noindent
{\bf Case 4}: The fourth case is words starting and ending with $*$. Let $i,j$ be the positions of the first and last zeros, respectively. The weight of the part to the right of the last zero is then $2M_{n-j-1} \left( \frac{1}{2} \right)$, and by the symmetry in Proposition~\ref{prop:ctdt-symm} the weight of the part to the left of the first zero is $2M_{i-2} \, \left( \frac{1}{2} \right)$. 
This gives
\begin{align*} 
&\sum_{i=2}^{n-n_0}\sum_{j=n_0+i-1}^{n-1} 2M_{i-2} \left( \frac{1}{2} \right) \, C_{j-i-1-(n_0-2)}^{j-i-1+(n_0-2)+1} 
\, 2M_{n-j-1} \left( \frac{1}{2} \right)\\ = 
&\sum_{i=2}^{n-n_0}\sum_{j=n_0+i-1}^{n-1} \binom{2i-2}{i-1} C_{j-i-1-(n_0-2)}^{j-i-1+(n_0-2)+1} \binom{2n-2j}{n-j}\\ = &\sum_{i=2}^{n-n_0} \binom{2i-2}{i-1} \binom{2(n-i)}{n-i-n_0+1}-\sum_{i=2}^{n-n_0} \binom{2i-2}{i-1} C_{n-i-n_0+1}^{n-i+n_0-2},
\end{align*} 
where the last equality is from Lemma \ref{L:Cat2}. To simplify the last sum for $n > n_0$ (both sums are 0 for $n = n_0$), we use again the same lemma which, after a suitable variable substitution, gives 
\[
\sum_{i=2}^{n-n_0} \binom{2i-2}{i-1} C_{n-i-n_0+1}^{n-i+n_0-2} 
=\binom{2n-2}{n-n_0}-C_{n-n_0}^{n+n_0-3}-\binom{2n-2n_0}{n-n_0}.
\]
Adding this result to those of Cases 1, 2 and 3, we get 
\begin{align*}
Z_{n,n_0} =&\sum_{i=1}^{n-n_0+1} \binom{2i-2}{i-1} \binom{2n-2i}{n-i-n_0+1}-\binom{2(n-n_0+1)-2}{n-n_0} \\
- &2\binom{2n-2}{n-n_0}
+2C_{n-n_0}^{n+n_0-3}+\binom{2n-2n_0}{n-n_0}+4\binom{2n-3}{n-n_0-1}.
\end{align*}
The last terms cancel and the theorem is proved for $n_0\ge 2$. Note that for $n = n_0$ the only contribution comes from Case 1 which gives the required result.

For $n_0=1$, Case 1 is empty but Case 2 and 3 are still $2\binom{2n-3}{n-2}=\binom{2n-2}{n-1}$ each. Case 4 is the same sum, with $i=j$, which becomes
$\sum_{i=2}^{n-1}\binom{2i-2}{i-1} \binom{2n-2i}{n-i}$. Together the cases give the desired formula which simplifies to $4^{n-1}$.

When $n_0=0$ this is given by Lemma \ref{L:Mk}, and that completes the proof.
\end{proof}

\begin{remark} 
One may rewrite the sum in Theorem \ref{T:partitionD} in several different ways, for example $Z_{n,n_0} =\sum_{j=0}^{n-n_0}\binom{2j-1}{j}\binom{2n-2j-1}{n-j-n_0}$, or as a recursion
$Z_{n,n_0} =4Z_{n-1,n_0} +C_{n-n_0}^{n+n_0-3}$. We do not need it here, but we note that the recursion is useful for proving a formula for a
generating function of the partition function:
\[
\sum_{n=n_0}^\infty Z_{n,n_0} t^{n}=\frac{t^{n_0}}{1-4t}\left(\frac{1-\sqrt{1-4t}}{2t}\right)^{2n_0-2}.
\] 
\end{remark}

Fix $n$ and $n_0$ such that $n_0 \leq n$.
Let, as before, ${\aver {i}_{n,n_0} }$ denote the stationary probability of having an $i$
Similarly, let ${\aver {i,j}_{n,n_0}}$ be the probability of having $i,j$ in the last two sites. 

\begin{theorem}  
\label{T:Dprob} 
For any $n\ge n_0\ge 1$, we have the following table for values of $Z_{n,n_0}\cdot\aver {a,b}_{n,n_0}$
 in the last two positions of the \dt{}.
\begin{center}
\renewcommand{\arraystretch}{2.5} 
\begin{tabular}{|c | c | c | }
\hline
$a\backslash b$ & $1\ \mathrm{or}\ \1$ & $0$ \\ 
\hline
$\1$ & $\sum_{j=2}^{n-n_0}\binom{2j-2}{j}\binom{2n-2j-2}{n-j-n_0}$ & $\displaystyle \binom{2n-3}{n-n_0-1}$ \\
\hline
$0$ & $\displaystyle \binom{2n-4}{n-n_0-1}$ &$ \displaystyle \binom{2n-4}{n-n_0}$ \\
\hline
$1$ & $\sum_{j=1}^{n-1-n_0}\binom{2j}{j}\binom{2n-2j-4}{n-j-n_0-1}$ & $\displaystyle \binom{2n-4}{n-n_0-1}$  \\
\hline
\end{tabular}
\end{center}
Note that having a $1$ in the last position has the same probability as having a $\1$.
\end{theorem}

\begin{proof} 
Using the cases in the proof of Theorem \ref{T:partitionD}, 
we first compute $\aver{ 0}_{n,n_0} =\binom{2n-2}{n-n_0}$ and 
\[
\aver{ \1}_{n,n_0} =\sum_{j=0}^{n-n_0}\binom{2j}{j}\binom{2n-2j-2}{n-j-n_0}-\binom{2n-2}{n-n_0}.
\] 
Clearly $\aver{ 0,0}_{n,n_0} ,\aver{ 1,0}_{n,n_0} $ and $\aver{ 0,\1}_{n,n_0} $ can be directly deduced from $\aver{ 0}_{n,n_0} $ with an appropriate shift of $n$ and $n_0$. Then use $\aver{ \1,0}_{n,n_0} =\aver{ 0}_{n,n_0} -\aver{ 0,0}_{n,n_0} -\aver{ 1,0}_{n,n_0} $. Note that $\aver{ 1,\1}_{n,n_0} $ must (because of $\beta = 1/2$) be twice $\aver{\1}_{n-1,n_0}$ and $\aver{ \1,\1}_{n,n_0} =\aver{ \1}_{n,n_0} -\aver{ 0,\1}_{n,n_0} -\aver{ 1,\1}_{n,n_0} $.
The equality of probabilities replacing $\1$ with $1$ in the last position follows from the lumping in Theorem \ref{thm:lump}.
 \end{proof}

\subsection{Limiting direction for type $\D$}
\label{sec:limD}

In this section, we will follow the strategy of Section~\ref{sec:limB} and prove formulas for the two-point correlation for the \dmt{}
using the two-point correlations in Section \ref{sec:bpart} for the \dt{}.
As before, let $\aver i_{D}$ and $\aver {i,j}_{D}$ mean the stationary probability  in the \dmt{} of having a particle of species $i$ at the last site and the last two sites, respectively. 
We will continue to use $Z_{n,n_0}$ for the partition function of the \dt{} with $n$ sites and $n_0$ $0$'s.

The limiting direction for the \dmt{} can now be determined from Theorem \ref{T:Dprob} in the same way as the limiting direction for the \bmt{} in Section \ref{sec:limB} since both have the same expression. Again, define
\begin{align*}
\Col_i(n) &:=\sum_{j=-n}^n \aver {j,i}_{D},\quad -n\le i\le n\\
\Row_i(n) &:=\sum_{j=-n}^n \aver {i,j}_{D},\quad -n\le i\le n \\
\Hd_i(n) &:=\sum_{j=i+1}^n \aver {i,-j}_{D}+ \aver {j,-i}_{D},\quad 1\le i\le n \\
\Hu_i(n) &:=\sum_{j=i+1}^n \aver {-j,i}_{D}+ \aver {-i,j}_{D},\quad 1\le i\le n.
\end{align*}

\begin{proof}[\textbf{Proof} of Theorem~\ref{thm:lim-dir-d}]
The formal expression for the limiting direction for type $D$ is \eqref{eq:limB}, that is, the same as that of type \hatB{}. This follows from that the highest root is the same and that inversions are 
computed in similar ways for $B_n$ and $D_n$. Thus the coefficient of the root $e_k$ in the limiting direction $\psi$ is again given by
\begin{multline}
\sum_{j=-k+1}^n \big( \aver {j,k}_{D} +\aver {k,j}_{D} \big) - \sum_{j=k+1}^n \big( \aver {j,-k}_{D} + \aver {-k,j}_{D} \big) \\
= \Row_k(n)-\Hd_k(n)+\Col_k(n)-\Hu_k(n).
\end{multline}
Plugging in the formulas in Lemma \ref{L:Corr-sums-D} for $k = 1$ and $2 \leq k \leq n$ proves the theorem. The case $n = 2$ can be checked separately.
\end{proof}

\begin{lemma}
\label{L:Corr-sums-D}
For any $n\ge 3$, we have in the \dmt{}
\begin{align*}
\Col_i(n) &= 
\begin{cases}
\frac{\binom{2n-2}{n-1}}{2^{2n-1}} & \abs{i} = 1,\\[0.2cm]
\frac{\binom{2n-2}{n-i}}{2Z_{n,i} } - \frac{\binom{2n-2}{n-i+1}}{2Z_{n,i-1} } & 1 < \abs{i} \le n.
\end{cases}
\end{align*}
\[
\Row_i(n) = 
\begin{cases}
\ds \frac{\binom{2n-4}{n-2}}{4^{n-1}}  & i = 1,\\[0.2cm]
\ds \frac{\binom{2n-3}{n-2}}{4^{n-1}} & i=-1,\\[0.2cm]
\ds \frac{2\sum_{j=2}^{n-i+1}\binom{2j-2}{j}\binom{2n-2j-2}{n-j-i+1} + \binom{2n-3}{n-i}}{Z_{n, i-1} } \\[0.2cm]
\quad - \ds \frac{2\sum_{j=2}^{n-i}\binom{2j-2}{j}\binom{2n-2j-2}{n-j-i}+\binom{2n-3}{n-i-1}}{Z_{n, i} } & -n \le i < -1,\\[0.2cm]
\ds \frac{2\sum_{j=1}^{n-i}\binom{2j}{j}\binom{2n-2j-4}{n-j-i} + \binom{2n-4}{n-i}}{Z_{n, i-1} }\\[0.2cm] 
\quad - \ds \frac{2\sum_{j=1}^{n-i-1}\binom{2j}{j}\binom{2n-2j-4}{n-j-i-1}+\binom{2n-4}{n-i-1}}{Z_{n, i} } & 1 < i \le n.
\end{cases}
\]
\[
\Hd_i(n) = 
\begin{cases}
\ds \frac{\binom{2n-4}{n-2}}{4^{n-1}}  & i=1,\\[0.2cm]
\ds \frac{\sum_{j=1}^{n-i}\binom{2j}{j}\binom{2n-2j-4}{n-j-i}}{Z_{n, i-1} } \\[0.2cm]
\quad - \ds\frac{\sum_{j=1}^{n-i-1}\binom{2j}{j}\binom{2n-2j-4}{n-j-i-1}}{Z_{n, i} }  & 2 \leq i \leq n.
\end{cases}
\]
\[
\Hu_i(n) = 
\begin{cases}
\ds \frac{\binom{2n-3}{n-2}}{4^{n-1}} & i=1,\\[0.2cm]
\ds \frac{\sum_{j=2}^{n-i+1}\binom{2j-2}{j}\binom{2n-2j-2}{n-j-i+1}}{Z_{n, i-1} } \\[0.2cm]
\quad - \ds\frac{\sum_{j=2}^{n-i}\binom{2j-2}{j}\binom{2n-2j-2}{n-j-i}}{Z_{n, i} }  & 2 \leq i \leq n.
\end{cases}
\]
\end{lemma}

\begin{proof} 
The proof is more or less the same computation as the proof of Lemma \ref{L:Corr-sums-B}, using Theorems \ref{T:partitionD} and \ref{T:Dprob} instead.
\end{proof}

\section*{Acknowledgements}

We thank Anne Schilling for useful discussions.
We thank Anne Schilling for useful discussions.
AA was partially supported by Department of Science and Technology grant EMR/2016/ 006624 and by the UGC Centre for Advanced Studies. SL and SP were supported by the Swedish Research Council grant 621-2014-4780.
This material is based upon work supported by the Swedish Research Council under grant no.~2016-06596 while AA, SL and SP were in residence at the Institut Mittag-Leffler in Djursholm, Sweden during the January-April semester in 2020.

\bibliographystyle{alpha}
\bibliography{asep}

\end{document}